\documentclass[12pt,lenq]{amsart}
\usepackage{amsmath}
\usepackage{amscd}
\usepackage{amssymb}
\usepackage{amsbsy}
\usepackage{amsfonts}
\usepackage{latexsym}
\usepackage{graphics}
\usepackage{amsmath,amscd,latexsym}
\usepackage{multirow}
\usepackage{array}
\usepackage{paralist}
\usepackage{titletoc}
\usepackage{epsfig, eucal,float}

\theoremstyle{plain}
\newtheorem{theorem}{Theorem}[section]
\newtheorem{proposition}[theorem]{Proposition}
\newtheorem{lemma}[theorem]{Lemma}
\newtheorem{key-lemma}[theorem]{Key Lemma}

\newtheorem{remark}[theorem]{Remark}
\newtheorem{definition}[theorem]{Definition}

\newtheorem{main theorem}[theorem]{Main Theorem}

\newtheorem{conjecture}[theorem]{Conjecture}

\newlength\savewidth

\newcommand{\interior}{\operatorname{int}}

\newcommand{\lk}{\operatorname{lk}}

\newcommand{\ZZ}{\mathbb{Z}}
\newcommand{\QQ}{\mathbb{Q}}
\newcommand{\RR}{\mathbb{R}}
\newcommand{\CC}{\mathbb{C}}
\newcommand{\HH}{\mathbb{H}}
\newcommand{\QQQ}{\hat{\mathbb{Q}}}
\newcommand{\RRR}{\hat{\mathbb{R}}}

\newcommand{\Conway}{\mbox{\boldmath$S$}^{2}}
\newcommand{\Conways}
{(\mbox{\boldmath$S$}^{2},\mbox{\boldmath$P$})}
\newcommand{\PP}{\mbox{\boldmath$P$}}
\newcommand{\PConway}{\mbox{\boldmath$S$}}
\newcommand{\ptorus}{\mbox{\boldmath$T$}}
\newcommand{\OO}{\mbox{\boldmath$O$}}

\newcommand{\rtangle}[1]{(B^3,t({#1}))}

\newcommand{\tr}{\mbox{$\mathrm{tr}$}}
\newcommand{\Isom}{\mbox{$\mathrm{Isom}$}}

\newcommand{\DD}{\mathcal{D}}
\newcommand{\RGPP}[1]{\hat\Gamma_{#1}}
\newcommand{\RGP}[1]{\Gamma_{#1}}
\newcommand{\curve}{\mathcal{C}}
\newcommand{\plamination}{\mathcal{PL}}
\newcommand{\einv}{\mathcal{E}}
\newcommand{\trg}{\mbox{$\mathrm{trg}$}}
\newcommand{\block}{\mbox{$\mathcal{B}$}}
\newcommand{\eblock}{\mbox{$\check{\mathcal{B}}$}}

\newcommand{\svert}{\,|\,}

\newcommand{\llangle}{\langle\langle}
\newcommand{\rrangle}{\rangle\rangle}

\newcommand{\inn}[1]{\ensuremath{\begin{aligned}[b] &{}\\[-.7cm] &{}\mkern3mu\scriptscriptstyle{\circ}\\[.15cm]
                                  \end{aligned}}\mkern-12mu#1}

    \newcommand{\ee}{\vec{e}}
    \newcommand{\EE}{\vec{E}}
    \newcommand{\TT}{\mathcal{T}}

\newcommand{\ABL}[1]{\mathcal{L}(#1)}
\newcommand{\RBL}[2]{\mathcal{L}(#1,#2)}

\begin{document}

\title{A variation of McShane's identity for 2-bridge links}

\author{Donghi Lee}
\address{Department of Mathematics\\
Pusan National University \\
San-30 Jangjeon-Dong, Geumjung-Gu, Pusan, 609-735, Republic of Korea}
\email{donghi@pusan.ac.kr}

\author{Makoto Sakuma}
\address{Department of Mathematics\\
Graduate School of Science\\
Hiroshima University\\
Higashi-Hiroshima, 739-8526, Japan}
\email{sakuma@math.sci.hiroshima-u.ac.jp}

\subjclass[2000]{Primary 57M25, 20F06 \\
\indent {The first author was supported by Basic Science Research Program
through the National Research Foundation of Korea(NRF) funded
by the Ministry of Education, Science and Technology(2009-0065798).
The second author was supported
by JSPS Grants-in-Aid 22340013 and 21654011.}}


\begin{abstract}
We give a variation of McShane's identity,
which describes the cusp shape of a hyperbolic $2$-bridge link
in terms of the complex translation lengths of
simple loops on the bridge sphere.
We also explicitly determine the set of end invariants of
$SL(2,\CC)$-characters of the once-punctured torus
corresponding to the holonomy representations
of the complete hyperbolic structures of 2-bridge link complements.
\end{abstract}
\maketitle

\begin{center}
{\it Dedicated to Professor Caroline Series
on the occasion of her 60th birthday}
\end{center}

\tableofcontents

\section{Introduction}
In his Ph.D. thesis ~\cite{McShane0},
McShane proved the following remarkable identity
concerning the lengths of simple closed geodesics
on a once-punctured torus
with a complete hyperbolic structure of finite area
(see also \cite{Bowditch1, Nakanishi}):
\begin{equation}
\label{McShane_identity}
\sum_{\beta\in \mathcal{C}}\frac{1}{1+e^{l(\beta)}}=\frac{1}{2}.
\end{equation}
Here $\mathcal{C}$ denotes the set of all simple closed geodesics on
a hyperbolic once-punctured torus, and
$l(\beta)$ denotes the length of a closed geodesic $\beta$.
This identity has been generalized
to cusped hyperbolic surfaces by McShane himself ~\cite{McShane},
to hyperbolic surfaces with cusps and geodesic boundary
by Mirzakhani ~\cite{Mirzakhani},
and to hyperbolic surfaces with cusps, geodesic boundary
and conical singularities
by Tan, Wong and Zhang ~\cite{Tan_Wong_Zhang_2}.
A wonderful application to the Weil-Petersson volume of the
moduli spaces of bordered hyperbolic surfaces was found by
Mirzakhani ~\cite{Mirzakhani}.
Bowditch ~\cite{Bowditch2} showed
that the identity (\ref{McShane_identity})
is also valid for all quasifuchsian punctured torus groups
where $l(\beta)$ is regarded
as the complex translation length of
the hyperbolic isometry corresponding to the closed geodesic $\beta$.
Moreover, he gave a nice variation of the identity
for hyperbolic once-punctured torus bundles,
which describes the cusp shape in terms of the
complex translation lengths of simple loops
on the fiber torus ~\cite{Bowditch3}.
Other $3$-dimensional variations have been obtained by
\cite{AMS, AMS2, Tan_Wong_Zhang_1, Tan_Wong_Zhang_2, Tan_Wong_Zhang_4,
Tan_Wong_Zhang_5, Tan_Wong_Zhang_6, Tan_Wong_Zhang_7}.

The main purpose of this paper is to prove yet another $3$-dimensional variation
of McShane's identity,
which describes the cusp shape of a hyperbolic $2$-bridge link
in terms of the complex translation lengths of
essential simple loops on the bridge sphere
(Theorems ~\ref{MainTheorem1} and \ref{MainTheorem2}).
This proves a conjecture proposed by the second author in \cite{Sakuma}.
The proof of the main results
is applied to the study of
the end invariants
of $SL(2,\CC)$-characters of the once-punctured torus
introduced by Bowditch ~\cite{Bowditch2} and
Tan, Wong and Zhang ~\cite{Tan_Wong_Zhang_1}.
In fact, we explicitly determine the sets of end invariants of
$SL(2,\CC)$-characters of the once-punctured torus
corresponding to the holonomy representations
of the complete hyperbolic structures of $2$-bridge link complements
(Theorems ~\ref{MainTheorem3} and ~\ref{MainTheorem4}).

This paper is organized as follows.
In Section ~\ref{Statement}, we give an explicit statement of the main result.
In Section ~\ref{sec:orbifold}, we recall basic facts
concerning the orbifold fundamental group of the
$(2,2,2,\infty)$-orbifold,
which connects the once-punctured torus
and the $4$-punctured sphere.
In Section ~\ref{sec:Markoff},
we recall basic facts concerning the type-preserving
$PSL(2,\CC)$-representations of the orbifold fundamental group.
In Section ~\ref{sec:canonical},
we describe a certain natural triangulation
of the cusps of hyperbolic $2$-bridge complements,
following \cite{Gueritaud, Sakuma-Weeks}.
In Section ~\ref{sec:proof},
we give a proof of the main result.
In Section ~\ref{sec:longitude},
we give an explicit homological description
of the longitudes of $2$-bridge links
in the definition of the cusp shapes in
Theorems ~\ref{MainTheorem1} and \ref{MainTheorem2}.
In the final section, Section ~\ref{sec:end_invarinat},
we applied the proof of the main results
to the study of the set of end invariants.

The authors would like to thank
Hirotaka Akiyoshi, Brian Bowditch,
Toshihiro Nakanishi, Caroline Series
and Ser Peow Tan for stimulating conversations.

\section{Statement of the main result}
\label{Statement}

Let $\ptorus$, $\PConway$ and $\OO$, respectively,
be the once-punctured torus,
the 4-times punctured sphere, and
the $(2,2,2,\infty)$-orbifold
(i.e., the orbifold with underlying space
a once-punctured sphere and with three cone points
of cone angle $\pi$).
They have $\RR^2-\ZZ^2$ as a common covering space.
To be precise,
let $H$ and $\tilde H$, respectively,
be the groups of transformations on $\RR^2-\ZZ^2$
generated by $\pi$-rotations about points in
$\ZZ^2$ and $(\frac{1}{2}\ZZ)^2$.
Then
$\ptorus=(\RR^2-\ZZ^2)/\ZZ^2$,
$\PConway=(\RR^2-\ZZ^2)/H$ and
$\OO=(\RR^2-\ZZ^2)/\tilde H$.
In particular, there are a $\ZZ_2$-covering
$\ptorus\to \OO$ and
a $\ZZ_2\oplus \ZZ_2$-covering
$\PConway\to \OO$:
the pair of these coverings is called
the {\it Fricke diagram}
and each of $\ptorus$, $\PConway$, and $\OO$
is called a {\it Fricke surface}
(see \cite{Sheingorn}).

A simple loop in a Fricke surface
is said to be {\it essential}
if it does not bound a disk,
a disk with one puncture, or
a disk with one cone point.
Similarly, a simple arc in a Fricke surface
joining punctures is said to be {\it essential}
if it does not cut off a ``monogon'', i.e.,
a disk minus a point on the boundary.
Then the isotopy classes of essential simple loops
(essential simple arcs with one end in a given puncture, respectively)
in a Fricke surface are in one-to-one correspondence with $\QQQ:=\QQ\cup\{1/0\}$:
a representative of the isotopy class corresponding to
$r\in\QQQ$ is the projection of a line in $\RR^2$
(the line being disjoint from $\ZZ^2$ for the loop case,
and intersecting $\ZZ^2$ for the arc case).
The element $r\in\QQQ$ associated to a loop or an arc
is called its {\it slope}.
An essential simple loop of slope $r$ in $\ptorus$ or $\OO$
is denoted by $\beta_r$,
while that in $\PConway$ is denoted by $\alpha_r$.
The notation reflects the following fact:
after an isotopy,
the restriction of the projection $\ptorus\to\OO$
to $\beta_r$ ($\subset \ptorus$)
gives a homeomorphism from $\beta_r$ ($\subset \ptorus$)
to $\beta_r$ ($\subset \OO$),
while the restriction of the projection
$\PConway\to \OO$ to $\alpha_r$
gives a two-fold covering from
$\alpha_r$ ($\subset \PConway$) to
$\beta_r$ ($\subset \OO$).

Now we recall the definition of a $2$-bridge link.
To this end,
set $\Conways=(\RR^2,\ZZ^2)/H$
and call it the {\it Conway sphere}.
Then $\Conway$ is homeomorphic to the 2-sphere,
$\PP$ consists of four points in $\Conway$, and
$\Conway-\PP$ is the $4$-punctured sphere $\PConway$.
We also call $\PConway$ the Conway sphere.
A {\it trivial tangle} is a pair $(B^3,t)$,
where $B^3$ is a 3-ball and $t$ is a union of two
arcs properly embedded in $B^3$
which is parallel to a union of two
mutually disjoint arcs in $\partial B^3$.
By a {\it rational tangle},
we mean a trivial tangle $(B^3,t)$
which is endowed with a homeomorphism from
$\partial(B^3,t)$ to $\Conways$.
Through the homeomorphism we identify
the boundary of a rational tangle with the Conway sphere.
Thus the slope of an essential simple loop in
$\partial B^3-t$ is defined.
We define
the {\it slope} of a rational tangle
to be the slope of
an essential loop on $\partial B^3 -t$ which bounds a disk in $B^3$
separating the components of $t$.
(Such a loop is unique up to isotopy
on $\partial B^3-t$ and so the slope of a rational tangle is well-defined.)

For each $r\in \QQQ$,
the {\it $2$-bridge link $K(r)$ of slope $r$}
is defined to be the sum of the rational tangles of slopes
$\infty$ and $r$, namely, $(S^3,K(r))$ is
obtained from $\rtangle{\infty}$ and $\rtangle{r}$
by identifying their boundaries through the
identity map on the Conway sphere $\Conways$. (Recall that the boundaries of
rational tangles are identified with the Conway sphere.)
$K(r)$ has one or two components according to whether
the denominator of $r$ is odd or even.
We call $\rtangle{\infty}$ and $\rtangle{r}$, respectively,
the {\it upper tangle} and {\it lower tangle} of the $2$-bridge link.

The topology of $2$-bridge links is nicely described by using
the {\it Farey tessellation} $\DD$, the tessellation of the hyperbolic
plane $\HH^2$ obtained from the ideal triangle
$\langle 0,1,\infty\rangle$ by successive reflection in its edges.
The vertex set of the Farey tessellation is equal to
$\QQQ:=\QQ\cup\{1/0\}\subset \partial \HH^2$
and so identified with the set of isotopy classes of essential simple loops
on a Fricke surface,
via the correspondence $s\mapsto \alpha_s$ or $\beta_s$.
Schubert's classification theorem of $2$-bridge links ~\cite{Schubert}
shows that two $2$-bridge links $K(r)$ and $K(r')$ are equivalent
(i.e., there is a self-homeomorphism of $S^3$ carrying one to the other)
if and only if there is an automorphism of $\DD$ which maps
the set $\{\infty, r\}$ to $\{\infty, r'\}$.

For each $r\in \QQQ$,
let $\RGP{r}$ be the group of automorphisms of
$\DD$ generated by reflections in the edges of $\DD$
with an endpoint $r$, and
let $\RGPP{r}$ be the group generated by $\RGP{r}$ and $\RGP{\infty}$.
Then the region, $R$, bounded by a pair of
Farey edges with an endpoint $\infty$
and a pair of Farey edges with an endpoint $r$
forms a fundamental domain of the action of $\RGPP{r}$ on $\HH^2$
(see Figure ~\ref{fig.fd}).
Let $I_1(r)$ and $I_2(r)$ be the closed intervals in $\RRR$
obtained as the intersection with $\RRR$ of the closure of $R$.
Suppose that $r$ is a rational number with $0<r<1$.
(We may always assume this except when we treat the
trivial knot and the trivial $2$-component link.)
Write
\begin{center}
\begin{picture}(230,70)
\put(0,48){$\displaystyle{
r=
\cfrac{1}{a_1+
\cfrac{1}{ \raisebox{-5pt}[0pt][0pt]{$a_2 \, + \, $}
\raisebox{-10pt}[0pt][0pt]{$\, \ddots \ $}
\raisebox{-12pt}[0pt][0pt]{$+ \, \cfrac{1}{a_n}$}
}} \
=:[a_1,a_2, \dots,a_n],}$}
\end{picture}
\end{center}
where $n \ge 1$, $(a_1, \dots, a_n) \in (\mathbb{Z}_+)^n$, and $a_n \ge 2$.
Then the above intervals are given by
$I_1(r)=[0,r_1]$ and $I_2(r)=[r_2,1]$,
where
\begin{align*}
r_1 &=
\begin{cases}
[a_1, a_2, \dots, a_{n-1}] & \mbox{if $n$ is odd,}\\
[a_1, a_2, \dots, a_{n-1}, a_n-1] & \mbox{if $n$ is even,}
\end{cases}\\
r_2 &=
\begin{cases}
[a_1, a_2, \dots, a_{n-1}, a_n-1] & \mbox{if $n$ is odd,}\\
[a_1, a_2, \dots, a_{n-1}] & \mbox{if $n$ is even.}
\end{cases}
\end{align*}

\begin{figure}[h]
\begin{center}
\includegraphics{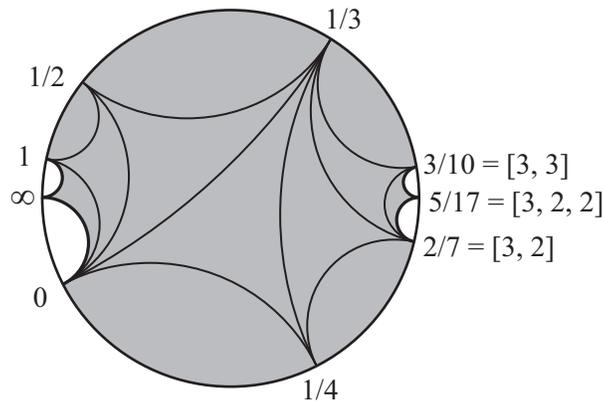}
\end{center}
\caption{\label{fig.fd}
A fundamental domain of $\hat\Gamma_r$ in the
Farey tessellation (the shaded domain) for $r=5/17=[3,2,2]$.}
\end{figure}

The following theorem was established by \cite{Ohtsuki-Riley-Sakuma}
and \cite{lee_sakuma}.

\begin{theorem}
\label{previous_results}
{\rm (1)} {\cite[Proposition ~4.6]{Ohtsuki-Riley-Sakuma}}
If two elements $s$ and $s'$ of $\QQQ$ belong to the same orbit $\RGPP{r}$-orbit,
then the unoriented loops $\alpha_s$ and $\alpha_{s'}$ are homotopic in $S^3-K(r)$.

{\rm (2)} {\cite[Lemma ~7.1]{lee_sakuma}}
For any $s\in\QQQ$,
there is a unique rational number
$s_0\in I_1(r)\cup I_2(r)\cup \{\infty, r\}$
such that $s$ is contained in the  $\RGPP{r}$-orbit of $s_0$.
In particular, $\alpha_s$ is homotopic to $\alpha_{s_0}$ in
$S^3-K(r)$.
Thus if $s_0\in\{\infty, r\}$, then $\alpha_s$ is null-homotopic
in $S^3-K(r)$.

{\rm (3)} {\cite[Main Theorem ~2.3]{lee_sakuma}}
The loop $\alpha_s$ is null-homotopic in $S^3 - K(r)$
if and only if $s$ belongs to the $\RGPP{r}$-orbit of $\infty$ or $r$.
In particular, if $s\in I_1(r)\cup I_2(r)$, then
$\alpha_s$ is not null-homotopic in $S^3-K(r)$.
\end{theorem}

Moreover, it is proved by \cite{lee_sakuma_2, lee_sakuma_3, lee_sakuma_4}
that generically two simple loops $\alpha_s$ and $\alpha_{s'}$
with $s,s'\in I_1(r)\cup I_2(r)$ are homotopic in the link complement
only when $s=s'$ (see Theorem ~\ref{prop:conjugacy}).

Now assume $r=q/p$, where $p$ and $q$ are relatively prime
positive integers such that $q\not\equiv \pm1 \pmod{p}$.
This is equivalent to the condition that $K(r)$ is hyperbolic,
namely the link complement $S^3-K(r)$ admits
a complete hyperbolic structure of finite volume.
Let $\rho_r$ be the $PSL(2,\CC)$-representation of $\pi_1(\PConway)$
obtained as the composition
\[
\pi_1(\PConway) \to
\pi_1(\PConway)/ \llangle\alpha_{\infty},\alpha_r\rrangle
\cong
\pi_1(S^3-K(r))
\to
\Isom^+(\HH^3)
\cong
PSL(2,\CC),
\]
where the last homomorphism is the holonomy representation
of the complete hyperbolic structure of $S^3-K(r)$.
Since $\pi(S^3-K(r))$ is generated by two meridians,
$\rho_r(\pi_1(\PConway))$ is generated by two parabolic transformations.
Hence the hyperbolic manifold $S^3-K(r)$ admits an isometric $\ZZ/2\ZZ\oplus \ZZ/2\ZZ$-action
(see \cite[Section ~5.4]{Thurston} and Figure ~\ref{fig.knot_symmetry}), and so
the $PSL(2,\CC)$-representation $\rho_r$ of $\pi_1(\PConway)$
extends to that of $\pi_1(\OO)$.
Moreover, this extension is unique
(see \cite[Proposition ~2.2]{ASWY}).
So we obtain, in a unique way, a
$PSL(2,\CC)$-representation of $\pi_1(\ptorus)$ by restriction.
We continue to denote it by $\rho_r$.

\begin{figure}[h]
\begin{center}
\includegraphics{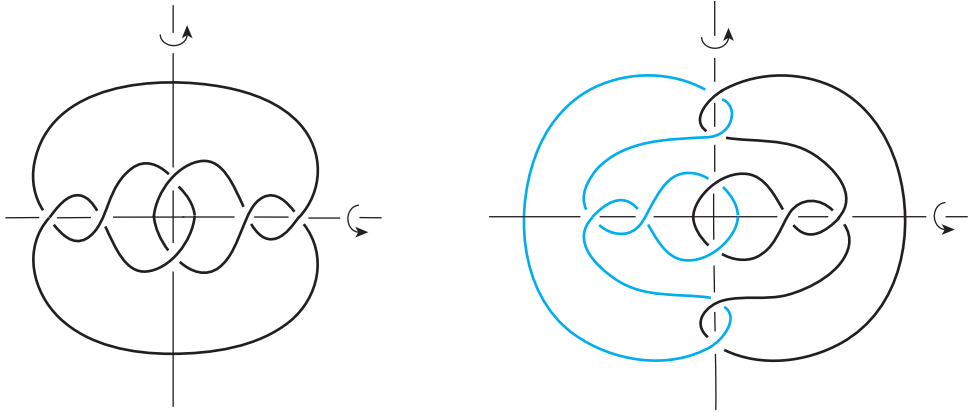}
\end{center}
\caption{\label{fig.knot_symmetry}
The symmetries of $S^3-K(r)$.}
\end{figure}

Note that each cusp of the hyperbolic manifold $S^3-K(r)$
carries a Euclidean structure, well-defined up to similarity,
and hence it is identified with the quotient of $\CC$
(with the natural Euclidean metric)
by the lattice $\ZZ\oplus \ZZ\lambda$,
generated by the translations
$[\zeta\mapsto \zeta+1]$ and $[\zeta\mapsto \zeta+\lambda]$
corresponding to the meridian and
a (suitably chosen) longitude, respectively.
This $\lambda$ does not depend on the choice of the cusp,
because when $K(r)$ is a two-component 2-bridge link
there is an orientation-preserving isometry of $S^3-K(r)$ interchanging the two cusps
(see Figure ~\ref{fig.knot_symmetry}).
We call $\lambda$ the {\it modulus} of the cusp and denote it by $\lambda(K(r))$.
In this paper, we prove the following variation of McShane's identity
which describes the modulus $\lambda(K(r))$
in terms of the complex translation lengths of
the images by $\rho_r$ of essential simple loops on $\ptorus$.

\begin{theorem}
\label{MainTheorem1}
For a hyperbolic $2$-bridge link $K(r)$,
the following identity holds:
\[
2\sum_{s\in \mathrm{int}I_1(r)}
\frac{1}{1+e^{l_{\rho_r}(\beta_s)}}
+
2\sum_{s\in \mathrm{int}I_2(r)}\frac{1}{1+e^{l_{\rho_r}(\beta_s)}}
+
\sum_{s\in\partial I_1(r)\cup\partial I_2(r)}\frac{1}{1+e^{l_{\rho_r}(\beta_s)}}
=-1,
\]
where $l_{\rho_r}(\beta_s)$ denotes the complex translation length of the hyperbolic isometry
$\rho_r(\beta_s)$.
Furthermore, the modulus $\lambda(K(r))$ of the cusp torus of the
cusped hyperbolic manifold $S^3-K(r)$
with respect to a suitable choice of a longitude
is given by the following formula:
\begin{align*}
\lambda(K(r))
&=
\frac{4}{|K(r)|}
\left\{
2\sum_{s\in \mathrm{int}I_1(r)}\frac{1}{1+e^{l_{\rho_r}(\beta_s)}}
+
\sum_{s\in\partial I_1(r)}\frac{1}{1+e^{l_{\rho_r}(\beta_s)}}
\right\}\\
&=
\frac{-4}{|K(r)|}
\left\{
2\sum_{s\in \mathrm{int}I_2(r)}\frac{1}{1+e^{l_{\rho_r}(\beta_s)}}
+
\sum_{s\in\partial I_2(r)}\frac{1}{1+e^{l_{\rho_r}(\beta_s)}}
+1
\right\},
\end{align*}
where $|K(r)|$ denotes the number of components of $K(r)$.
\end{theorem}

In the above theorem and in the remainder of this paper,
we abuse notation to use the symbol $\beta_s$ to denote
an element of $\pi_1(\ptorus)$ represented by
the simple loop $\beta_s$ with an arbitral orientation.
Though such an element of $\pi_1(\ptorus)$ is determined
only up to conjugacy and inverse,
the complex translation length of $\rho_r(\beta_s)$
is uniquely determined by the slope $s$.
Here the {\it complex translation length} of an orientation-preserving
isometry of $\HH^3$ is an element of $\CC/2\pi i\ZZ$
whose real part is the translation length along the axis of the isometry
and whose imaginal part is the rotation angle around the axis.
If $\rho_r(\beta_s)$ is represented by a matrix $A\in SL(2,\CC)$,
then the complex translation length $l_{\rho_r}(\beta_s)$
is determined by the formula
$
\tr A=\cosh\left(\frac{l_{\rho_r}(\beta_s)}{2}\right)
$
and the condition that the real translation length
$L(\rho_r(\alpha_s)):=\Re(l_{\rho_r}(\beta_s))$
is non-negative.

At the end of this section,
we restate the above theorem in terms of
the hyperbolic $3$-orbifold $\OO(r):=S^3-K(r)/(\ZZ/2\ZZ\oplus \ZZ/2\ZZ)$,
the quotient of the hyperbolic manifold $S^3-K(r)$
by the isometric $\ZZ/2\ZZ\oplus \ZZ/2\ZZ$-action.
To this end, note that the action extends to an action on $(S^3,K(r))$
satisfying the following conditions (see Figure ~\ref{fig.knot_symmetry}).
\begin{enumerate}[\indent \rm (1)]
\item If we regard $S^3$ as the one-point compactification of $\RR^3$,
then the action on $S^3$ consists of the $\pi$-rotations around
the three coordinate axes.
In particular, the singular set, $F$, is the union of the three axes and the point at infinity.

\item If $K(r)$ is a knot, then
both of the $\ZZ/2\ZZ$-factors act on $K(r)$ as reflections,
and so $K(r)$ intersects $F$ in $4$-points.
Each of the four sub-intervals of $K(r)$ divided by
$K(r)\cap F$
forms a fundamental domain of the restriction of the action to $K(r)$.

\item If $K(r)$ has two components,
then one of the $\ZZ/2\ZZ$-factors interchanges the components of $K(r)$
and the other factor preserves both components of $K(r)$
and acts on each component as a reflection.
In particular, each component of $K(r)$ intersects $F$
in two points, and each of the four sub-intervals of $K(r)$ divided by
$K(r)\cap F$ forms a fundamental domain of the restriction of the action to $K(r)$.
\end{enumerate}
Hence the orbifold $\OO(r)$ has a single cusp,
which forms a Euclidian orbifold of type $(2,2,2,2)$,
i.e, the orbifold with underlying space $S^2$ and with four cone points of cone angle $\pi$.
Recall that the cusp torus of $S^3-K(r)$ is
identified with the quotient of $\CC$
by the lattice $\ZZ\oplus \ZZ\lambda$,
generated by the translations
$[\zeta\mapsto \zeta+1]$ and
$[\zeta\mapsto \zeta+\lambda]$
corresponding to the meridian and
a (suitably chosen) longitude, respectively.
By the above description of the $\ZZ/2\ZZ\oplus \ZZ/2\ZZ$-action,
we see that the cusp of $\OO(r)$ is the quotient of $\CC$
by the group generated by $\pi$-rotations around the origin $0$,
the point $\frac{1}{2}$, and the point $\frac{|K(r)|}{4}\lambda$.
It should be noted that
the line segment in $\CC$ joining $0$ and $\frac{|K(r)|}{4}\lambda$
projects homeomorphically onto a simple arc joining two cone points in
$\partial\OO(r)$, whose inverse image in $S^3-K(r)$ forms
a longitude if $K(r)$ is a knot;
it forms a union of longitudes of the two components
if $K(r)$ is a $2$-component link.
We call the simple arc in $\partial\OO(r)$ a {\it longitude}.
We set
\[
\lambda(\OO(r)):=\frac{\frac{|K(r)|}{4}\lambda}{\frac{1}{2}}=
\frac{|K(r)|}{2}\lambda,
\]
and call it
the {\it modulus} of the cusp of $\OO(r)$
with respect to the longitude.
Then Theorem ~\ref{MainTheorem1} is paraphrased as follows.

\begin{theorem}
\label{MainTheorem2}
For a hyperbolic $2$-bridge link $K(r)$,
the modulus of the cusp of the quotient orbifold $\OO(r)$
with respect to a suitable choice of a longitude
is given by the following formula:
\begin{align*}
\frac{1}{2}\lambda(\OO(r))
&=
2\sum_{s\in \mathrm{int}I_1(r)}\frac{1}{1+e^{l_{\rho_r}(\beta_s)}}
+
\sum_{s\in\partial I_1(r)}\frac{1}{1+e^{l_{\rho_r}(\beta_s)}}\\
&=
-2\sum_{s\in \mathrm{int}I_2(r)}\frac{1}{1+e^{l_{\rho_r}(\beta_s)}}
-
\sum_{s\in\partial I_2(r)}\frac{1}{1+e^{l_{\rho_r}(\beta_s)}}
-1.
\end{align*}
\end{theorem}

An explicit description of
the longitude in Theorems ~\ref{MainTheorem1} and \ref{MainTheorem2}
is given by Proposition ~\ref{prop:longitude}.

\section{The orbifold fundamental group $\pi_1(\OO)$}
\label{sec:orbifold}

The orbifold fundamental group
of $\OO$ has the presentation:
\begin{equation}
\label{orbifold-group-presentation}
\pi_1(\OO) = \langle A, B, C\; |\;
A^2=B^2=C^2=1\rangle,
\end{equation}
where $D:=(ABC)^{-1}$ is represented by
the puncture of $\OO$.
We call $D$ the {\it distinguished element}.
Since $\ptorus$ and $\PConway$ are finite regular coverings of
the orbifold $\OO$,
the fundamental groups of $\ptorus$ and $\PConway$ are regarded
as normal subgroups of $\pi_1(\OO)$ of indices $2$ and $4$, respectively.

By picking a complete hyperbolic structure of $\OO$ (and hence
of $\ptorus$), we identify the universal covering space
$\tilde\OO=\tilde \ptorus$ with (the upper-half-space
model of) the hyperbolic plane \mbox{$\HH^2=\{z\in\CC\mid\Im(z)>0\}$,} and
identify $\pi_1(\OO)$ with a Fuchsian group
(see Figure ~\ref{fig.fuchsian-group}).
We assume that $D$ is identified with the following parabolic transformation
having the ideal point $\infty$ of $\HH^2$
as the parabolic fixed point:
\begin{equation*}
\label{NormalizingD0}
D(z)=z+1.
\end{equation*}
Then the points $A(\infty)$, $B(\infty)$ and $C(\infty)$ lie on $\RR$
from left to right in this order.
After a coordinate change, we may assume that
the images of the three geodesics
joining $\infty$ with these three points, in the
universal abelian cover $\RR^2-\ZZ^2$ of $\ptorus$,
are open arcs of slopes $0$, $1$ and $\infty$,
joining the puncture $(0,0)$ with $(1,0)$, $(1,1)$ and $(0,1)$, respectively.
Thus the images of these three geodesics in $\ptorus$
are mutually disjoint arcs properly embedded in $\ptorus$,
which divide $\ptorus$ into two ideal triangles, and thus
they determine an ideal triangulation of $\ptorus$.

\begin{figure}[h]
\begin{center}
\includegraphics{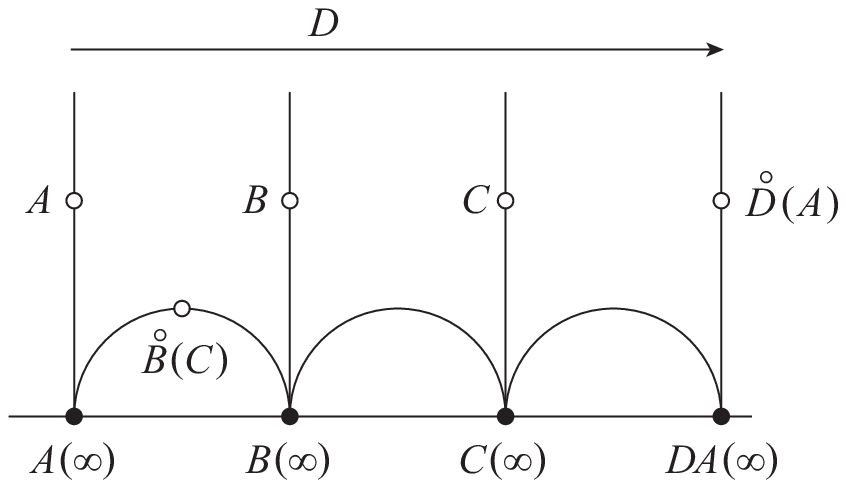}
\end{center}
\caption{\label{fig.fuchsian-group}
The Fuchsian group $\langle A,B,C\rangle$. The symbols $A$, $B$,
$C$, $\hskip1.5mm{}^{{}^{{}^\circ}}\hskip-3.2mm D(A)$ and $\hskip1.5mm{}^{{}^{{}^\circ}}\hskip-3.2mmB(C)$
are situated near the fixed points in $\HH^2$ of the involutions they denote.}
\end{figure}

We now recall the well-known correspondence between the ideal
triangulations of $\ptorus$ and the {\it Farey triangles},
i.e., a triangle in the Farey tessellation $\DD$.
The vertex set of $\DD$ is equal to $\QQQ:=\QQ\cup\{1/0\}
\subset \partial \HH^2$ and each vertex $s$ determines a properly
embedded arc $\delta_s$ in $\ptorus$ of {\it slope} $s$, i.e., the arc
in $\ptorus$ obtained as the image of the straight arc of slope $r$ in
$\RR^2-\ZZ^2$ joining punctures. If $\sigma=\langle
s_0,s_1,s_2\rangle$ is a Farey triangle,
then the arcs $\delta_{s_0}$, $\delta_{s_1}$ and
$\delta_{s_2}$ are mutually disjoint and they determine an ideal
triangulation of $\ptorus$. In the following we assume
that the orientation of $\sigma=\langle s_0,s_1,s_2\rangle$ is
coherent with the orientation of the Farey triangle $\langle
0,1,\infty\rangle$, where the orientation is determined by the order
of the vertices. Then the oriented simple loop in $\ptorus$ around the
puncture representing $D^2$ meets the edges of the ideal triangulation $\trg(\sigma)$
of slopes $s_0,s_1,s_2$ in this cyclic order, for every Farey triangle
$\sigma=\langle s_0,s_1,s_2\rangle$.

By using the above notation,
the generators $A$, $B$ and $C$
in (\ref{orbifold-group-presentation}) are described as follows.
Consider the ideal triangulation $\trg(\sigma)$ of $\ptorus$
determined by the Farey triangle $\sigma=\langle 0,1,\infty\rangle$.
It lifts to a $\pi_1(\OO)$-invariant tessellation of
the universal cover $\tilde \ptorus=\HH^2$.
Let $\{e_j\}_{j\in\ZZ}$ be the edges of the tessellation emanating from
the ideal vertex $\infty$, lying in $\HH^2$ from left to right in this order.
For each $e_j$,
there is a unique order $2$ element, $P_j$, in $\pi_1(\OO)$
which inverts $e_j$.
We may assume after a shift of indices that
$e_{3j}$, $e_{3j+1}$ and $e_{3j+2}$
project to the arcs in $\ptorus$ of slopes $0$, $1$ and $\infty$,
respectively, for every $j\in\ZZ$.
Then any triple of consecutive elements
$(P_{3j}, P_{3j+1}, P_{3j+2})$ serves as $(A,B,C)$.
Throughout this paper, $(A,B,C)$ represents the triple of
specific elements of $\pi_1(\OO)$ obtained in this way.
We call $\{P_j\}_{j\in\ZZ}$
the {\it sequence of elliptic generators}
associated with the Farey triangle $\sigma$.

The above construction works for every Farey triangle
$\sigma=\langle s_0,s_1,s_2\rangle$,
and the sequence of elliptic generators associated with it is defined.
(Here we use the assumption that the orientation of
$\langle s_0,s_1,s_2\rangle$ is coherent with the orientation of
$\langle 0,1,\infty\rangle$.)
Any triple of three consecutive elements
in a sequence of elliptic generators is called an
{\it elliptic generator triple}.
A member, $P$, of an elliptic generator triple is called
an {\it elliptic generator},
and its {\it slope} $s(P)\in\QQQ$ is defined
to be the slope of the arc in $\ptorus$ obtained as the image
of the geodesic $\langle \infty, P(\infty)\rangle$.
(Here it should be noted that $\infty$ is the parabolic fixed point
of the distinguished element $D$.)
For example, we have
\[
(s(A),s(B),s(C))=(0,1,\infty).
\]
When we say that $\{P_j\}_{j\in\ZZ}$
is the sequence of elliptic generators
associated with a Farey triangle
$\sigma=\langle s_0,s_1,s_2\rangle$,
we always assume that
\[
(s(P_{3m}),s(P_{3m+1}),s(P_{3m+2}))=(s_0,s_1,s_2).
\]
Thus the index $j$ is well defined modulo a shift by a multiple of
$3$. We summarize the properties of elliptic generators
(cf. ~\cite[Section ~2.1]{ASWY}),
by using the following non-standard notation
which was introduced in \cite{Dicks-Sakuma}.
\begin{equation}
\label{eq:conjugation}
\text{For elements $X$, $Y$ of a group $G$, $\inn{X}(Y)$ denotes $XYX^{-1}$.}
\end{equation}
We view $\inn{X}$ as an element of the automorphism group of $G$.

\begin{proposition}
{\rm (1)}  Let $\{P_j\}_{j\in\ZZ}$ be the sequence of elliptic generators
associated with a Farey triangle $\sigma$.
Then the following hold for every $j\in\ZZ$.
\begin{enumerate}[\indent \rm (i)]
\item $\pi_1(\OO) \cong \langle P_j, P_{j+1}, P_{j+2}\; |\;
P_j^2=P_{j+1}^2=P_{j+2}^2=1\rangle$.

\item $P_{j+2}P_{j+1}P_j$ is equal to the distinguished element $D$
of $\pi_1(\OO)$.

\item With the notation of \eqref{eq:conjugation}, $P_{j+3m}= \inn{D^m}(P_j)$
for every $m \in \ZZ$.

\item $\langle s(P_j), s(P_{j+1}), s(P_{j+2})\rangle$
is a Farey triangle and its orientation is coherent with
$\langle 0,1,\infty\rangle$.
\end{enumerate}

{\rm (2)} Let $P$ and $P'$ be elliptic generators of the same slope.
Then $P'= \inn{D^m}(P)$ for some $m\in\ZZ$.
Let $\sigma=\langle s_0,s_1,s_2 \rangle$ and
$\sigma'=\langle s_0',s_1',s_2' \rangle$ be Farey triangles
sharing the edge
$\langle s_0,s_1\rangle = \langle s_0',s_2' \rangle$,
and let $\{P_j\}$ and $\{P_j'\}$, respectively,
be the sequences of elliptic generators associated with
$\sigma$ and $\sigma'$.
Then the following identity holds after
a shift of indices by a multiple of $3$
{\rm (}see Figure ~{\rm\ref{fig.3b})}.
\[
(P'_{3j},P'_{3j+1},P'_{3j+2})=(P_{3j},\inn{P_{3j+1}}(P_{3j+2}),P_{3j+1}).
\]
\end{proposition}

\begin{figure}[h]
\begin{center}
\includegraphics{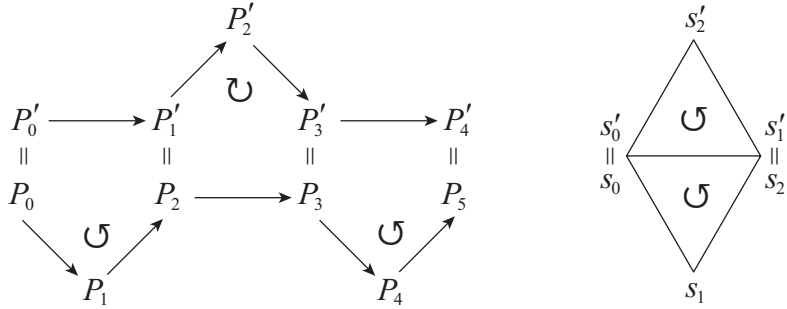}
\end{center}
\caption{\label{fig.3b}
Adjacent sequences of elliptic generators.
The symbol $\circlearrowleft$ ($\circlearrowright$, respectively) indicates
a triangle in which coherent reading of the vertices is counter-clockwise (clockwise, respectively).}
\end{figure}

The above proposition motivates us to introduce the following definition,
which is used in the description of a cusp triangulation of
the $2$-bridge link complement.

\begin{definition}
[Chain]
\label{def.chain2}
{\rm
By a {\it chain} of Farey triangles, we mean a (non-empty) finite sequence
$\Sigma=(\sigma_1,\dots,\sigma_m)$ of mutually distinct Farey triangles
such that $\sigma_{i+1}$ is adjacent to $\sigma_{i}$ for each $i$ ($1\le i\le m-1$).
}	
\end{definition}

\begin{definition}
[Elliptic generator complex]
\label{def.LL2}
{\rm
Let $\Sigma=(\sigma_1,\dots,\sigma_m)$ be a chain of Farey triangles.

(1) $\ABL{\Sigma}$ (or $\ABL{\sigma_1,\dots,\sigma_m}$)
denotes the simplicial complex constructed as follows,
and we call it the {\it elliptic generator complex}
\index{elliptic generator complex}
associated with the chain $\Sigma$:
\begin{enumerate}[\indent \rm (i)]
\item The vertex set $\ABL{\Sigma}^{(0)}$ is identified with
the set of elliptic generators whose slope is contained in $\Sigma^{(0)}$.

\item The edge set $\ABL{\Sigma}^{(1)}$ is identified with
the set of the ordered pairs $(P,Q)$ of elliptic generators,
which appears (successively in this order) in a sequence of elliptic generators
associated with a triangle of $\Sigma$.

\item The set $\ABL{\Sigma}^{(2)}$ of the $2$-simplices
is identified with the set of the elliptic generator triples $(P,Q,R)$
such that $(P,Q)$, $(Q,R)$ and $(P,R)$ are edges of $\ABL{\Sigma}$.
\end{enumerate}

(2) The self-map $P\mapsto \inn{D}(P)$ on $\ABL{\Sigma}^{(0)}$
induces a simplicial automorphism on $\ABL{\Sigma}$,
and we denote it by the symbol $D$.

(3) $\ABL{\Sigma}/\langle D\rangle$ and $\ABL{\Sigma}/\langle D^2\rangle$,
respectively, denote the abstract cell complex obtained as
the quotient of $\ABL{\Sigma}$ by the group
$\langle D\rangle$ and $\langle D^2\rangle$.
}
\end{definition}

The 1-skeleton of $\ABL{\Sigma}$ is obtained as the union of
the $1$-dimensional simplicial complex $\ABL{\sigma_i}$
($\sigma_i\in\Sigma^{(2)}$),
which is obtained by joining the vertices $\{P_j\}$,
the sequence of elliptic generators
associated with $\sigma_i$, successively by edges.

\section{Markoff maps and type-preserving representations}
\label{sec:Markoff}

In this section,
we recall basic facts concerning type-preserving
$PSL(2,\CC)$-representations of $\pi_1(\OO)$,
and explain a key proposition, Proposition ~\ref{key-proposition},
which was established by Bowditch ~\cite{Bowditch2} and
generalized by \cite{AMS} and \cite{Tan_Wong_Zhang_4}.

A $PSL(2,\CC)$-representation of $\pi_1(\ptorus)$ or $\pi_1(\OO)$
is {\it type-preserving} if it is irreducible
(equivalently, it does not have a common fixed point in $\partial\HH^3$)
and sends peripheral elements to parabolic transformations.
When we mention a type-preserving representation
$\rho:\pi_1(\OO)\to PSL(2,\CC)$, we always assume that the
image of the distinguished element $D$ is given by
\begin{equation}
\label{tp-normalization}
\rho(D)
=\begin{pmatrix}1&1\\0&1\end{pmatrix}.
\end{equation}
Since $\pi_1(\ptorus)$ is a free group,
any type-preserving representation
$\rho:\pi_1(\ptorus)\to PSL(2,\CC)$
lifts to a presentation
$\tilde\rho:\pi_1(\ptorus)\to SL(2,\CC)$,
which is {\it type-preserving},
in the sense that it is irreducible
and sends peripheral elements to parabolic transformations.
For a type-preserving $SL(2,\CC)$-representation
$\tilde\rho:\pi_1(\ptorus)\to SL(2,\CC)$,
let $\phi=\phi_{\tilde{\rho}}$ be the map from
$\DD^{(0)}=\QQQ$ to $\CC$ defined by
$\phi(s)=\mathrm{tr}(\tilde{\rho}(\beta_s))$.
Then it is a nontrivial {\it Markoff map} in the sense of \cite{Bowditch2}, that is:
\begin{enumerate}[\indent \rm (i)]
\item For any Farey triangle $\langle s_0,s_1,s_2 \rangle$,
the triple $(\phi(s_0),\phi(s_1),\phi(s_2))$ is a
nontrivial {\it Markoff triple}, that is,
it is a nontrivial solution of the Markoff equation
\begin{equation*}
x^2+y^2+z^2=xyz.
\end{equation*}
Here, being nontrivial means $(x,y,z)\ne (0,0,0)$.

\item For any pair of triangles $\langle s_0,s_1,s_2 \rangle$ and
$\langle s_1,s_2,s_3\rangle$ of $\DD$ sharing a common edge $\langle s_1,s_2\rangle$,
we have
\begin{equation*}
\phi(s_0)+\phi(s_3)=\phi(s_1)\phi(s_2).
\end{equation*}
\end{enumerate}

\begin{lemma}
\label{lem.trace}
Let $\rho:\pi_1(\OO)\to PSL(2,\CC)$ be a type-preserving representation
satisfying the normalization condition \eqref{tp-normalization},
and let $\phi$ be a Markoff map induced by a type-preserving representation
$\tilde\rho:\pi_1(\ptorus)\to SL(2,\CC)$
which is a lift of the restriction of $\rho$ to $\pi_1(\ptorus)$.

{\rm (1)}
Let $P$ be an elliptic generator of slope $s$.
\begin{enumerate}[\indent \rm (i)]
\item If $\phi(s)\ne 0$, then $\rho(P)$ is the $\pi$-rotation about the geodesic
with endpoints $c(\rho(P))\pm i/\phi(s)$,
where $c(\rho(P))=\rho(P)(\infty)\in\CC$.

\item If $\phi(s)=0$, then $\rho(P)$ is the $\pi$-rotation about a vertical geodesic,
i.e., a geodesic in the upper-half space model of the hyperbolic space
which has $\infty$ as an endpoint.
\end{enumerate}

{\rm (2)}
Let $\{P_j\}$ be the sequence of elliptic generators
associated with a Farey triangle $\sigma=\langle s_0,s_1,s_2 \rangle$.
\begin{enumerate}[\indent \rm (i)]
\item If $\phi(s_1)\phi(s_2)\ne 0$, then
\begin{equation*}
c(\rho(P_{2}))-c(\rho(P_{1}))
=\frac{\phi(s_{0})}{\phi(s_{1})\phi(s_{2})}.
\end{equation*}

\item If $\phi(s_0)\phi(s_1)\phi(s_2)\ne 0$, then
$\{c(\rho(P_j))\}$ is a sequence of points in $\CC$
such that $c(\rho(P_j))+1=c(\rho(P_{j+3}))$.

\item If $\phi(s_0)=0$, then $\phi(s_1)\phi(s_2)\ne 0$ and
$c(\rho(P_{1+3k}))=c(\rho(P_{2+3k}))$
and $\rho(P_{3k})$ is a $\pi$-rotation about a vertical geodesic,
where the end point of its axis in $\CC$ is the midpoint
between $c(\rho(P_{-2+3k}))=c(\rho(P_{-1+3k}))$
and $c(\rho(P_{1+3k}))=c(\rho(P_{2+3k}))$.

\item Suppose $\phi(s_0')=0$,
where $s_0'$ is the vertex of $\DD$ opposite to $s_0$
with respect to the edge $\langle s_1, s_2\rangle$.
Then $\phi(s_0)\phi(s_1)\phi(s_2)\ne 0$ and
$c(\rho(P_{-1+3k}))=c(\rho(P_{1+3k}))$.
\end{enumerate}
\end{lemma}

\begin{proof}
The assertions of the lemma except for (2-iv)
are contained in \cite[Proposition ~2.2.4]{ASWY}.
To prove (2-iv), note that the assumption $\phi(s_0')=0$
implies that $\phi(s_2)=\pm i \phi(s_1)$ and $\phi(s_0)=\pm i \phi(s_1)^2$.
Since $(\phi(s_0'),\phi(s_1),\phi(s_2))$ is nontrivial,
this implies that $\phi(s_1)\ne 0$ and so
none of $\phi(s_j)$ $(j\in\{0,1,2\}$) is $0$.
\end{proof}

The following definition
is used in a description of cusp triangulations of
hyperbolic $2$-bridge link complements.

\begin{definition}
\label{def.LL}
{\rm
Let $\rho:\pi_1(\OO)\to PSL(2,\CC)$ be a type-preserving representation
satisfying the normalization condition \eqref{tp-normalization}.

(1) Let $\sigma=\langle s_0,s_1,s_2 \rangle$ be
a Farey triangle such that $\phi(s_0)\phi(s_1)\phi(s_2)\ne 0$.
Then $\RBL{\rho}{\sigma}$ denotes the (possibly singular) bi-infinite
zigzag line in $\CC$
which is obtained by successively joining
the points $\{c(\rho(P_j))\}$,
where $\{P_j\}$ is the sequence of elliptic
generators associated with $\sigma$.
We note that $\RBL{\rho}{\sigma}$
is invariant by the transformation $z\mapsto z+1$.

(2) We say that $\RBL{\rho}{\sigma}$ is {\it simple},
if the underlying space $|\RBL{\rho}{\sigma}|$
is homeomorphic to the real line $\RR$
and $\{c(\rho(P_j))\}$
sits on it in the order of the suffix $j\in\ZZ$.

(3) We say that $\RBL{\rho}{\sigma}$
is {\it simply folded} at the vertex $c(\rho(P_j))$,
if $c(\rho(P_{j-1}))=c(\rho(P_{j+1}))$ and
the horizontal line, $L$, passing through this point does not contain
$c(\rho(P_j))$.
In this case,
the underling space $|\RBL{\rho}{\sigma}|$ is the union of $L$
and the ``spikes'' joining $c(\rho(P_{j+3k}))$ and
$c(\rho(P_{j-1+3k}))=c(\rho(P_{j+1+3k}))\in L$,
where $k$ runs over all integers.
We call $L$ the {\it horizontal line} determined by $\RBL{\rho}{\sigma}$.
If $s$ is the slope of the elliptic generator $P_j$,
we also say that $\RBL{\rho}{\sigma}$ is {\it simply folded} at the slope $s$.
We say that $\RBL{\rho}{\sigma}$
is {\it simply folded} if it is simply folded at some slope
(which is a vertex of $\sigma$).

(4) Let $\Sigma=(\sigma_1,\sigma_2,\dots,\sigma_m)$ be a chain of Farey triangles
such that $\phi^{-1}(0)\cap \Sigma^{(0)}=\emptyset$. Then $\RBL{\rho}{\Sigma}$
denotes the union of the bi-infinite zigzag lines $\{\RBL{\rho}{\sigma_i}\}$ in $\CC$.
}
\end{definition}

\begin{remark}
\label{remark:simply-folded}
{\rm
By Lemma ~\ref{lem.trace}(2),
$\RBL{\rho}{\sigma}$ is simply folded at the slope $s$
only when $\phi(s')=0$,
where $s'$ is the vertex of $\DD$ opposite to $s$
with respect to the edge of $\sigma$ which does not contain $s$.
}
\end{remark}

Let $\TT$ be a binary tree
(a countably infinite simplicial tree
all of whose vertices have degree $3$)
properly embedded in $\HH^{2}$ dual to $\DD$.
A {\it directed edge}, $\ee$,
of $\TT$ can be thought of an ordered pair
of adjacent vertices of $\TT$,
referred to as the {\it head} and {\it tail} of $\ee$.
Following \cite{Bowditch2},
we use the notation
$\ee \leftrightarrow (s_1,s_2;s_0,s_3)$
to mean that
$s_0$, $s_1$, $s_2$ and $s_3$
are the ideal vertices of $\DD$ such that
\begin{enumerate}[\indent \rm (i)]
\item the Farey edge $\langle s_1,s_2\rangle$ is
the dual to $\ee$, and

\item the Farey triangle $\langle s_0,s_1,s_2 \rangle$
($\langle s_1,s_2,s_3\rangle$, respectively) is dual to the head
(tail, respectively) of $\ee$, if $\ee \leftrightarrow (s_1,s_2;s_0,s_3)$.
\end{enumerate}
If $\phi(s_1)\phi(s_2)\ne 0$, then we set
\[
\psi(\ee):=\frac{\phi(s_0)}{\phi(s_1)\phi(s_2)}.
\]
Then Lemma ~\ref{lem.trace}(2-i) is rephrased as follows:
\begin{equation}
\label{a=x/yz}
c(\rho(P_{2}))-c(\rho(P_{1}))=\psi(\ee).
\end{equation}
We regard $\psi=\psi_{\phi}$ as a map from the set
of oriented edges $\ee \leftrightarrow (s_1,s_2;s_0,s_3)$
of $\TT$ such that $\phi(s_1)\phi(s_2)\ne 0$, and we call it
the {\it complex probability map}
associated with the Markoff map $\phi$.
We note that this map is determined by
the type-preserving representation
$\rho:\pi_1(\ptorus)\to PSL(2,\CC)$ obtained from
the type-preserving $SL(2,\CC)$-representation of
$\pi_1(\ptorus)$ inducing the Markoff map $\phi$.
So we also call $\psi$ the complex probability map
associated with $\rho$.

By a {\it complementary region} of $\TT$, we mean the
closure of a connected component of $\HH^2-\TT$.
Let $\Omega$ be the set
of complementary regions of $\TT$.
Then there is a natural bijection from $\Omega$ to $\QQQ$.
In the following we identify $\Omega$ with $\QQQ$.
Let $\ee \leftrightarrow (s_1,s_2;s_0,s_3)$ be a directed edge of $\TT$.
If we remove the interior of $e$ from $\TT$,
we are left with two disjoint subset,
which we denote by $\TT^{\pm}(\ee)$,
so that $e\cap\TT^{+}(\ee)$ is the head of $\ee$ and
$e\cap\TT^{-}(\ee)$ is its tail.
Let $\Omega^{\pm}(\ee)\subset\Omega$
be the set of regions whose boundaries lie in
$\Sigma^{\pm}(\ee)$, and set $\Omega^0(e)=\{s_1,s_2\}$.
We see that $\Omega$ can be written as the disjoint union:
$\Omega=\Omega^0(e)\cup\Omega^+(\ee)\cup\Omega^-(\ee)$.
Set $\Omega^{0-}(\ee)=\Omega^0(e)\cup\Omega^-(\ee)$
and $\Omega^{0+}(\ee)=\Omega^0(e)\cup\Omega^+(\ee)$.
We quote the following key proposition from
\cite[Proposition ~5.2]{AMS} which is
a slight extension of \cite[Proposition ~3.13]{Bowditch2}
(see \cite{Tan_Wong_Zhang_4} for further extension).

\begin{proposition}
\label{key-proposition}
Let $\phi$ be a Markoff map and $\ee$ a directed edge of $\TT$
satisfying the following conditions.
\begin{enumerate}[\indent \rm (i)]
\item The set $\{s\in\Omega^-(\ee)\,|\,\ |\phi(s)|\le 2\}$ is
finite.

\item $\Omega^{0-}(\ee)\cap\phi^{-1}(-2,2)=\emptyset$.
\end{enumerate}
Then
\[
\psi(\ee)=\sum_{s\in\Omega^0(e)} h(\phi(s))
+2\sum_{s\in\Omega^-(\ee)} h(\phi(s)).
\]
Moreover, the above sum converges absolutely.
\end{proposition}

Here, $h:\mathbb{C}-[-2,2] \to \mathbb{C}$ is defined by
$h(x)=\frac{1}{2}\left(1-\sqrt{1-4/x^2}\right)$,
where we adopt the convention that the real
part of a square root is always non-negative.
For each $s\in\Omega=\QQQ$,
let $l(\rho(\beta_s))$ be the complex translation length
of the isometry $\rho(\beta_s)$ of $\mathbb{H}^3$,
where we abuse notation to denote by $\beta_s$
an element of $\pi_1(\ptorus)$ represented by
the simple loop $\beta_s$ of slope $s$.
Then we have the following (see \cite[p.721]{Bowditch2}):
\[
h(\phi(s))=\frac{1}{1+e^{l(\rho(\beta_s))}}.
\]

At the end of this section,
we give a necessary and sufficient condition
for a type-preserving $PSL(2,\CC)$-representation
to descend to a representation of the 2-bridge link group $\pi_1(S^3-K(r))$.

\begin{lemma}
\label{2-bridge-markoff}
Let $\phi$ be a nontrivial Markoff map,
and let $\rho:\pi_1(\OO)\to PSL(2,\CC)$ be a type-preserving representation
induced by $\phi$.
Then the restriction of $\rho$ to $\PConway$ descends to a
representation of the $2$-bridge link group $\pi_1(S^3-K(r))$
if and only if $\phi(\infty)=\phi(r)=0$.
\end{lemma}

\begin{proof}
Though this is proved in \cite[Proof of Proposition ~4.1]{Sakuma2},
we give a proof for completeness.
Since $\pi_1(S^3-K(r))=\pi_1(\PConway)/
\langle\langle\alpha_{\infty},\alpha_r\rangle\rangle$,
a representation $\rho$ descends to a representation of
the link group if and only if $\rho(\alpha_{\infty})=\rho(\alpha_{r})=1$.
Since $\alpha_{s}=\beta_s^2$,
this condition is equivalent to the condition that
either $\rho(\beta_s)$ is the identity or an elliptic transformation of order $2$
for each $s=\infty$ and $r$.
However, since $\rho$ is irreducible by the assumption,
we have $\rho(\beta_s)\ne 1$ for any $s\in\QQQ$.
Thus the above condition is equivalent to the condition
that both $\rho(\beta_{\infty})$ and $\rho(\beta_r)$ are elliptic of order $2$,
which in turn is equivalent to the condition $\phi(\infty)=\phi(r)=0$.
\end{proof}

\section{The canonical decomposition of $S^3-K(r)$
and the induced cusp triangulation}
\label{sec:canonical}

In this section, we describe the canonical
decompositions of hyperbolic $2$-bridge link complements
and the induced cusp triangulations,
following \cite{Gueritaud, Sakuma-Weeks}.
Let $K(r)$ be a hyperbolic $2$-bridge link.
Then we may assume $r=q/p$, where $p$ and $q$ are relatively prime integers
such that $2\le q < p/2$, and so $r$ has the continued fraction expansion
$[a_1,a_2, \dots, a_n]$, where
$(a_1, \dots, a_n) \in (\mathbb{Z}_+)^n$, $a_1\ge 2$, $a_n \ge 2$ and $n\ge 2$.
Set $c=\sum_{i=1}^n a_i$, and let $\Sigma(r)=(\sigma_1,\sigma_2,\dots,\sigma_c)$
be the chain of Farey triangles
which intersect the hyperbolic geodesic joining $\infty$ with $r$ in this order
(see Figure ~\ref{fig.farey_triangle_sequence}).

\begin{figure}[h]
\begin{center}
\includegraphics{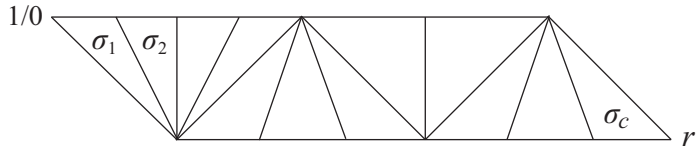}
\end{center}
\caption{\label{fig.farey_triangle_sequence}
The chain $\Sigma(r)$ of Farey triangles.}
\end{figure}

Just as with the once-punctured torus $\ptorus$,
each Farey triangle
$\sigma=\langle s_0,s_1,s_2\rangle$ determines
a (topological) ideal triangulation
of the $4$-times punctured sphere $\PConway=(\RR^2-\ZZ^2)/H$.
To be precise, the union of the lines in $\RR^2-\ZZ^2$
passing through the punctures of slopes $\{s_0,s_1,s_2\}$
determines an $H$-invariant ideal triangulation
of $\RR^2-\ZZ^2$,
and this descends to an ideal triangulation of $\PConway$.
The $1$-skeleton of this ideal triangulation
consists of three pairs of edges
corresponding to the three vertices $\{s_0,s_1,s_2\}$
of $\sigma$.
In the remainder of this paper,
we abuse notation and use the symbol $\trg(\sigma)$
to denote this ideal triangulation of $\PConway$.

If $\sigma$ and $\sigma'$ are adjacent Farey triangles,
then $\trg(\sigma')$ is obtained from $\trg(\sigma)$ by
a ``diagonal exchange'',
i.e., deleting a pair of edges of slope $s$ and adding a pair of edges of slope $s'$,
where $s$ ($s'$, respectively) is the vertex of $\sigma$
($\sigma'$, respectively) which is not
contained in $\sigma'$ ($\sigma$, respectively).
As illustrated in Figure ~\ref{fig.diagonal_exchange},
$\trg(\sigma)$ and $\trg(\sigma')$ can be regarded as the bottom and top faces
of an immersed pair of topological ideal tetrahedra in $\PConway\times \RR$,
which we denote by $\trg(\sigma, \sigma')$.

\begin{figure}[h]
\begin{center}
\includegraphics{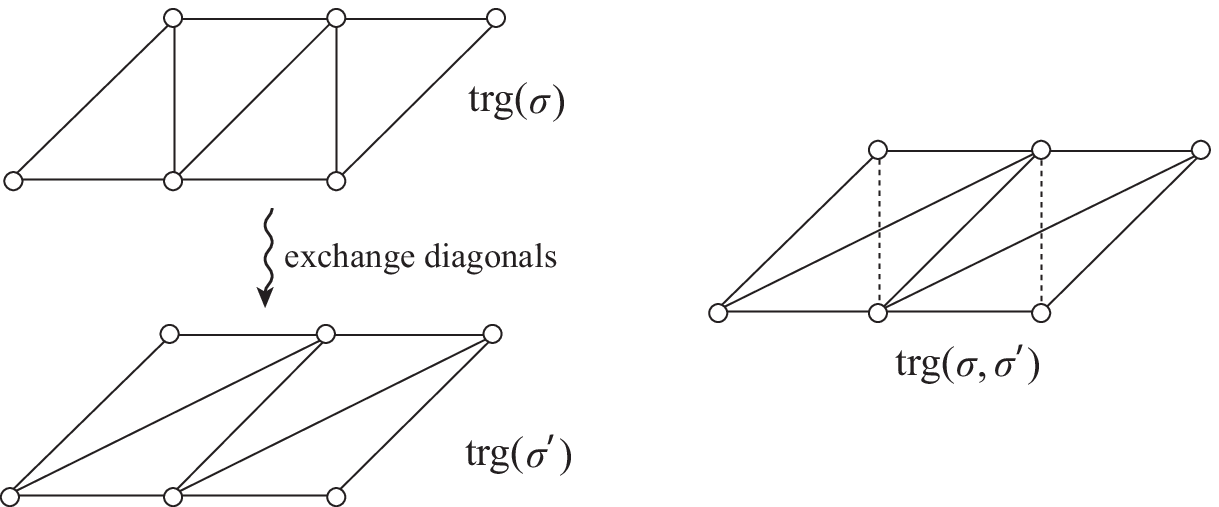}
\end{center}
\caption{\label{fig.diagonal_exchange}
A diagonal exchange of an ideal triangulation of $\PConway$
determines an immersed pair
of ideal tetrahedra in $\PConway\times\RR$.}
\end{figure}

The immersed topological pairs of ideal tetrahedra
$\{\trg(\sigma_i,\sigma_{i+1})\}_{1\le i\le c-1}$
can be stacked up to form a topological ideal triangulation, $\hat\DD(r)$, of
$\PConway\times [-1,1]$.
The restriction of $\hat\DD(r)$ to $\PConway\times \{-1\}$
($\PConway\times \{+1\}$, respectively)
is $\trg(\sigma_1)$ ($\trg(\sigma_c)$, respectively),
and each $\trg(\sigma_i)$ can be regarded as (a triangulation of)
a pleated surface in $\PConway\times [-1,1]$.
Let $\DD(r)$ be the topological ideal simplicial complex obtained from
$\hat\DD(r)$
by collapsing each edge of slope $\infty$ and $r$ into an ideal vertex.
To be precise, $\DD(r)$ is constructed as follows.
Since each edge of slope $\infty$ is collapsed into an ideal vertex,
the subcomplex $\trg(\sigma_1)$ of $\hat\DD(r)$
is collapsed into a single ideal edge, and
$\trg(\sigma_2)$ is folded along the pair of edges of slope $1/2$
to a pair ideal triangles
as illustrated in Figure ~\ref{fig.collapse}.
(Note that the slope $1/2$ is
the vertex of $\sigma_2$ which is not contained in $\sigma_1$.)
Similarly, since each edge of slope $r$ is collapsed into an ideal vertex,
the subcomplex $\trg(\sigma_c)$ of $\hat\DD(r)$
is collapsed into a single ideal edge, and
$\trg(\sigma_{c-1})$ is folded along the pair of edges of slope
$[a_1,\dots,a_{n}-2]$
into a pair of ideal triangles.
(Note that the slope $[a_1,\dots,a_{n}-2]$ is
the vertex of $\sigma_{c-1}$ which is not contained in $\sigma_c$.)
In other words, $\DD(r)$ is obtained from the subcomplex
$\hat\DD_0(r):=\{\trg(\sigma_i,\sigma_{i+1})\}_{2\le i\le c-2}$ of $\hat\DD(r)$
by folding the bottom surface $\trg(\sigma_2)$
in the pair of edges of slope $1/2$
and by folding the top surface $\trg(\sigma_{c-1})$
in the pair of edges of slope $[a_1,\dots,a_{n}-2]$,
as described in \cite[p.408]{Sakuma-Weeks}.
Hence $\DD(r)$ gives a topological ideal triangulation of $S^3-K(r)$
by \cite[Theorem ~II.2.4]{Sakuma-Weeks}.

\begin{figure}[h]
\begin{center}
\includegraphics{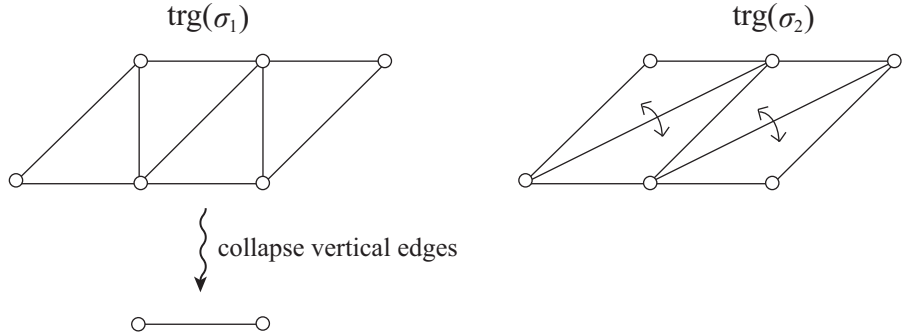}
\end{center}
\caption{\label{fig.collapse}
The effect of the collapsing of the edges of slope $\infty$
in the ideal triangulation $\hat\DD(r)$ of $\PConway\times [-1,1]$:
The subcomplex $\trg(\sigma_1)$ is collapsed into a single ideal edge,
and the subcomplex $\trg(\sigma_2)$ is folded along the edges of slope $1/2$
into a pair of ideal triangles.
}
\end{figure}

It should be noted that
the edges of $\trg(\sigma_2)$ of slopes $0/1$ and $1/1$
are identified into a single edge in $\DD(r)$,
which forms the ``core tunnel''
of the rational tangle $(B^3,t(\infty))$ in $(S^3,K(r))$
(see \cite[Figure ~II.2.7]{Sakuma-Weeks}
and \cite[Figure ~17]{Gueritaud}).
Similarly, the edges of $\trg(\sigma_{c-1})$ of slopes $[a_1,\dots,a_{n-1}]$ and
$[a_1,\dots,a_{n}-1]$
are identified into a single edge in $\DD(r)$,
which forms the core tunnel of the rational tangle $(B^3,t(r))$
in $(S^3,K(r))$.

Now we describe the triangulation of the peripheral torus of $S^3-K(r)$
induced by $\DD(r)$.
To this end, we identify the underlying space
of the subcomplex $\hat\DD_0(r)$ of $\hat\DD(r)$ with $\PConway\times [-1,1]$,
and we first describe the triangulation of the peripheral annuli
of $\PConway\times [-1,1]$ induced by $\hat\DD_0(r)$.
Since the combinatorics of the four peripheral annuli are identical,
let us focus on a single peripheral annulus, $A$.
Since $\trg(\sigma_i)$ is an ideal triangulation of a level $4$-punctured sphere,
it induces a triangulation, $C(\sigma_i)$, of a core circle in $A$.
The triangulation $C(\sigma_i)$ consists of $3$ vertices and $3$ edges.
By recalling the definition of the slopes of elliptic generators,
we may identify $C(\sigma_i)$
with the quotient complex $\ABL{\sigma_i}/\langle D\rangle$
(see Definition ~\ref{def.LL2}(3)).
The region in $A$ bounded by $C(\sigma_i)$ and $C(\sigma_{i+1})$
consists of $2$ triangles
as illustrated in Figure ~\ref{fig.local_triangulation},
and the triangulation of the region can be identified with
$\ABL{\sigma_i,\sigma_{i+1}}/\langle D\rangle$
(compare Figure ~\ref{fig.local_triangulation} with Figure ~\ref{fig.3b}).
The family $\{C(\sigma_i)\}_{2\le i\le c-1}$ forms the $1$-skeleton of the triangulation of $A$, which
is identified with  $\ABL{\Sigma_0(r)}/\langle D\rangle$,
where $\Sigma_0(r):=(\sigma_2,\dots,\sigma_{c-1})$.

\begin{figure}[h]
\begin{center}
\includegraphics{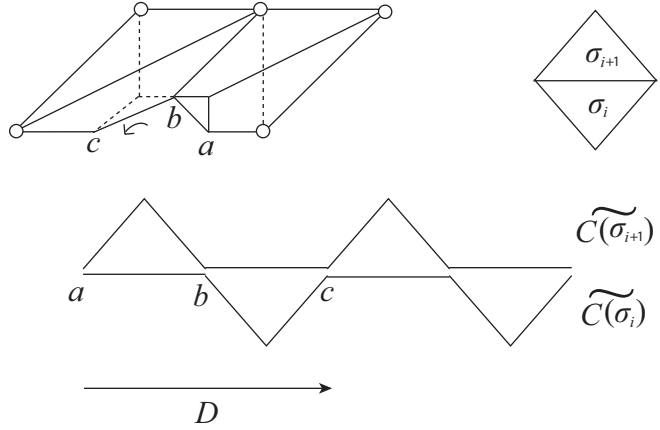}
\end{center}
\caption{\label{fig.local_triangulation}
Local picture of the triangulation of the peripheral annuli of
$\PConway\times [-1,1]$.
In the infinite cyclic cover of the peripheral annulus $A$,
the inverse images of $C(\sigma_i)$ and $C(\sigma_{i+1})$
form periodic zigzag lines, and the triangles bounded by them
project to the two triangles in $A$.
}
\end{figure}

Next, we explain the effect, to the triangulation $\ABL{\Sigma_0(r)}/\langle D\rangle$ of
the peripheral annulus $A$,
of the folding of the pleated surfaces $\trg(\sigma_2)$ and $\trg(\sigma_{c-1})$.
To this end, let $\{P_j\}$ be the sequence of elliptic generators
associated with $\sigma_2=\langle 0/1,1/2,1/1\rangle$
such that $(s(P_0),s(P_1),s(P_2))=(0/1,1/2,1/1)$.
Since the edges of slope $0/1$ and $1/1$ in $\hat\DD(r)$
are identified into a single edges by the folding of $\trg(\sigma_2)$
along the edges of slope $1/2$,
the vertices $[P_0]$ and $[P_2]$ of $\ABL{\Sigma_0(r)}/\langle D\rangle$
are identified.
In the infinite cyclic cover $\ABL{\Sigma_0(r)}$,
the boundary line $\ABL{\sigma_1}$
is deformed into a zigzag line which has a ``hairpin curve''
at the vertices $P_{1+3k}$,
where the vertices $P_{3k}$ and $P_{2+3k}$ are identified into a single vertex
for each $k\in\ZZ$
(see Figure \ref{fig.folding_cusp_triangulation}).
Furthermore, since the folding joins the punctures of $\PConway\times [-1,1]$,
the resulting triangulation of the peripheral annulus $A$ is joined to
the corresponding triangulation of another peripheral annuls
as illustrated in Figure ~\ref{fig.folding_cusp_triangulation}
(see \cite[Figure ~19]{Gueritaud}).
Similarly, the folding of the pleated surface $\trg(\sigma_{c-1})$
cause a similar effect on the other side of $A$.

\begin{figure}[h]
\begin{center}
\includegraphics{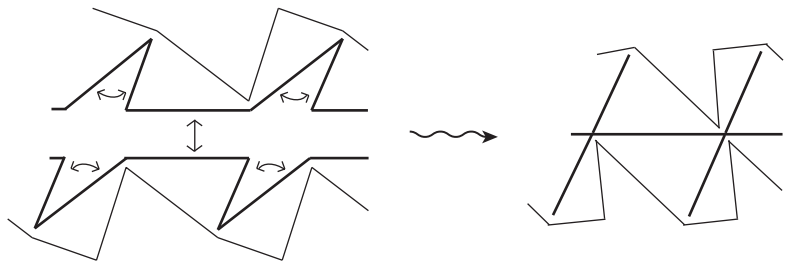}
\end{center}
\caption{\label{fig.folding_cusp_triangulation}
The effect of the folding in the cusp triangulation,
viewed in the infinite cyclic cover.
}
\end{figure}

In \cite{Gueritaud}, Futer applied Gueritaud's technique
based on angled structures to proved that
the topological ideal triangulation $\DD(r)$ is geometric, namely,
$\DD(r)$ is homeomorphic to
a geometric ideal triangulation of the complete hyperbolic manifold $S^3-K(r)$.
(Moreover, Gueritaud ~\cite{Gueritaud2} proved that
$\DD(r)$ is homeomorphic to the canonical decomposition of $S^3-K(r)$
in the sense of \cite{Epstein-Penner, Weeks},
proving the conjecture in \cite{Sakuma-Weeks}.
In the second author's joint work
with Akiyoshi, Wada and Yamashita ~\cite{ASWY},
an approach using cone manifold deformation
toward the same conclusion is announced.)
The geometric ideal triangulation $\DD(r)$ induces a geometric triangulation
of each cusp torus of $S^3-K(r)$.
Since $\DD(r)$ is preserved by the $\ZZ/2\ZZ\oplus \ZZ/2\ZZ$-action
described in Section ~\ref{Statement} (see Figure ~\ref{fig.knot_symmetry}),
this triangulation does not depend on a choice of a cusp,
and we call it {\it the triangulation of the cusp of $S^3-K(r)$
induced by the geometric ideal triangulation $\DD(r)$},
or simply {\it the cusp triangulation induced by $\DD(r)$}.
The cusp triangulation is preserved by the $\ZZ/2\ZZ\oplus\ZZ/2\ZZ$-action on $S^3-K(r)$,
and so it induces a ``triangulation'' of the cusp of the quotient orbifold $\OO(r)$.
We also call it {\it the triangulation of the cusp of $\OO(r)$
induced by the geometric ideal triangulation $\DD(r)$},
or simply {\it the cusp triangulation induced by $\DD(r)$}.
The following proposition describes the geometric structure of the cusp triangulation.
(See Figure ~\ref{fig.actual_cusp_triangulation}
for the actual geometric picture of the cusp triangulation,
which is produced by SnapPea.
See also Figure ~\ref{fig.actual_cusp_triangulation2}
which illustrates the periodic zigzag lines $\RBL{\rho_r}{\sigma_i}$
in the proposition.)

\begin{figure}[h]
\begin{center}
\includegraphics{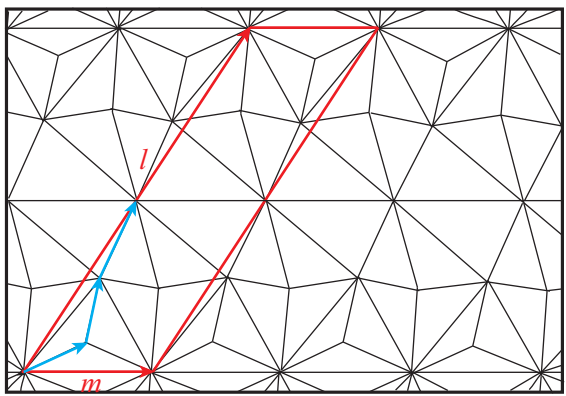}
\end{center}
\caption{\label{fig.actual_cusp_triangulation}
The actual cusp triangulation of $S^3-K([3,2,2])$.
The oriented zigzag line segment represent that
obtained by joining the points
$c(\rho_r(P_0)), c(\rho_r(P_1)), \dots, c(\rho_r(P_d))$
in the proof of Proposition ~\ref{prop.cusp-shape}.
}
\end{figure}

\begin{figure}[h]
\begin{center}
\includegraphics{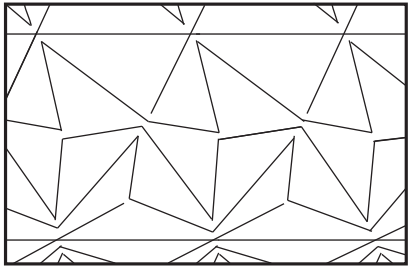}
\end{center}
\caption{\label{fig.actual_cusp_triangulation2}
The periodic zigzag lines
$\{\RBL{\rho_r}{\sigma_i}\}$
in the cusp triangulation.}
\end{figure}

\begin{proposition}
\label{geometric_shape_of_cusp}
Let $K(r)$ be a hyperbolic $2$-bridge link, and let
$\rho_r:\pi_1(\OO)\to PSL(2,\CC)$ be a type-preserving representation
induced by the holonomy representation of the complete hyperbolic structure of $S^3-K(r)$
satisfying the normalization condition \eqref{tp-normalization}.
Then the following hold.

{\rm (1)}
We have $\phi_r^{-1}(0)\cap \Sigma(r)^{(0)}=\{\infty,r\}$, where
$\Sigma(r)^{(0)}$ is the vertex set of $\Sigma(r)$ and
$\phi_r$ is the Markoff map induced by a lift of the restriction of
$\rho_r$ to $\pi_1(\ptorus)$ to an $SL(2,\CC)$-representation.

{\rm (2)}
The zigzag line $\RBL{\rho_r}{\sigma_2}$ is simply folded at the slope $1/2$;
$\RBL{\rho_r}{\sigma_{c-2}}$ is simply folded
at the slope $[a_1,a_2,\dots,a_{n}-2]$.

{\rm (3)}
Let $L_-$ and $L_+$ be the horizontal lines
determined by the simply folded zigzag lines
$\RBL{\rho_r}{\sigma_2}$ and $\RBL{\rho_r}{\sigma_{c-2}}$, respectively.
Then $\RBL{\rho_r}{\Sigma_0(r)}$ forms a $1$-skeleton of a triangulation
of the strip, $\tilde A$, in $\CC$ bounded by $L_-$ and $L_+$.
This triangulation descends to
the triangulation of the cusp of $\OO(r)$ induced by $\DD(r)$.
To be precise, the following hold.
\begin{enumerate}[\indent \rm (i)]
\item Let $P_-$ and $P_+$ be elliptic generators of slope $\infty$ and $r$,
respectively. Then $\rho_r(P_-)$ {\rm (}$\rho_r(P_+)$, respectively{\rm )} acts on $\CC$ as
the $\pi$-rotation about the center of an edge of $\RBL{\rho_r}{\sigma_2}$
{\rm (}$\RBL{\rho_r}{\sigma_{c-2}}$, respectively{\rm )} contained in $L_-$
{\rm (}$L_+$, respectively{\rm )}.
In particular, $\tilde A$ forms a fundamental domain
of the action on $\CC$ of the infinite dihedral group
generated by $\rho_r(P_-)$ and $\rho_r(P_+)$.

\item The orbifold fundamental group, $\pi_1(\partial \OO(r))$, of the cusp of $\OO(r)$
is identified with the group $\langle \rho_r(D), \rho_r(P_-), \rho_r(P_+)\rangle$.
Moreover, $\rho_r(D)$ corresponds to a meridian of $K(r)$,
whereas $\rho_r((P_+P_-)^2)$ or $\rho_r(P_+P_-)$ corresponds to
a longitude of $K(r)$
according to whether $K(r)$ has one or two components.

\item The images of the triangulation $\RBL{\rho_r}{\Sigma_0(r)}$ of $\tilde A$
by the infinite dihedral group $\langle \rho_r(P_-), \rho_r(P_+)\rangle$
form a $\pi_1(\partial \OO(r))$-invariant triangulation of $\CC$,
which project to the triangulation of the cusp of $\OO(r)$ induced by $\DD(r)$.
\end{enumerate}
\end{proposition}

\begin{proof}
(1) By Lemma ~\ref{2-bridge-markoff}, we have $\phi_r(\infty)=\phi_r(r)=0$.
Let $s$ be a vertex of $\Sigma_0(r)$.
Since simple arcs of slope $s$ in $\PConway$
is realized as a geodesic edge in the geometric triangulation
$\DD(r)$ of the hyperbolic manifold $S^3-K(r)$,
it follows that if $P$ is an elliptic generator of slope $s$
then $\rho_r(P)(\infty)\ne \infty$.
By Lemma ~\ref{lem.trace}(2), this implies that $\phi(s)\ne 0$.
Hence we have $\phi_r^{-1}(0)\cap \Sigma(r)^{(0)}=\{\infty,r\}$.

(2) This follows from the fact that $\phi_r(\infty)=\phi_r(r)=0$ and Lemma ~\ref{lem.trace}(2-iv)
(cf. Remark ~\ref{remark:simply-folded}).

(3) By the preceding description of the combinatorial structure of the cusp triangulation
and the fact that $\DD(r)$ is geometric,
we see that
$\RBL{\rho_r}{\Sigma_0(r)}$ froms a $1$-skeleton of a triangulation
of the strip, $\tilde A$, in $\CC$ bounded by $L_-$ and $L_+$.
Since $\phi_r(\infty)=\phi_r(r)=0$,
we see by Lemma ~\ref{lem.trace}(2-iii)
that $\rho_r(P_-)$ ($\rho_r(P_+)$, respectively) acts on $\CC$ as
the $\pi$-rotation about the center of an edge of
$\RBL{\rho_r}{\sigma_2}$ ($\RBL{\rho_r}{\sigma_{c-1}}$, respectively)
contained in $L_-$ ($L_+$, respectively).
It is obvious that $\tilde A$
forms a fundamental domain of the infinite dihedral group
$\langle \rho_r(P_-), \rho_r(P_+)\rangle$, and so we have (i).
Since $\tilde A$ projects to one of the four peripheral annuli of
$\PConway\times [-1,1]$
and since the $\ZZ/2\ZZ\oplus\ZZ/2\ZZ$-action on $S^3-K(r)$
lifts to a $\ZZ/2\ZZ\oplus\ZZ/2\ZZ$-action on $\PConway\times [-1,1]$
which acts transitively on the set of the four peripheral annuli,
we see that $\tilde A$ is a fundamental domain of the action of
$\pi_1(\partial\OO(r))$ modulo the action of the meridian $\rho_r(D)$.
Since $\tilde A$
is a fundamental domain of the infinite dihedral group
$\langle \rho_r(P_-), \rho_r(P_+)\rangle$,
we have
$\pi_1(\partial \OO(r))=\langle \rho_r(D), \rho_r(P_-), \rho_r(P_+)\rangle$.
Thus we obtain the first assertion of (ii).
The remaining assertion of (ii) follows from the description of
the $\ZZ/2\ZZ\oplus\ZZ/2\ZZ$-action on $K(r)$ given at the end of Section ~\ref{Statement}.
Assertion (iii) follows from (i), (ii),
and the description of the combinatorial structure of the cusp triangulation
together with the fact that $\DD(r)$ is geometric.
\end{proof}

\section{Proof of Theorems ~\ref{MainTheorem1} and \ref{MainTheorem2}}
\label{sec:proof}

Throughout this section and in the remainder of this paper,
$K(r)$ denotes a hyperbolic link,
$\rho_r:\pi_1(\OO)\to PSL(2,\CC)$ denotes
the type-preserving representation induced by the holonomy
representation of the complete hyperbolic structure of $S^3-K(r)$,
$\phi_r$ denotes a Markoff map determined by a lift
$\tilde\rho_r:\pi_1(\ptorus)\to SL(2,\CC)$ of the restriction of $\rho_r$
to $\pi_1(\ptorus)$,
and $\psi_r$ denotes the complex probability map
determined by $\rho_r$.

Let $\TT_0(r)$ be the subtree of $\TT$ dual to the chain $\Sigma_0(r)$,
and let $\EE(r)$ be the set of the oriented edges of $\TT-\TT_0(r)$
whose head is contained in  $\TT_0(r)$.
For each interval $I_j(r)$ ($j=1,2$),
we consider the following set of oriented edges:
\[
\EE_j(r)=\{\ee\in \EE(r) \svert \Omega^{0-}(\ee)\subset I_j(r)\}.
\]
It should be noted that
$\EE(r)=\EE_1(r)\sqcup\EE_2(r)\sqcup\{\ee_-,\ee_+\}$,
where $\ee_-$ and $\ee_+$ are the element of $\EE(r)$
with tails dual to $\sigma_1$ and $\sigma_c$, respectively
(see Figure ~\ref{fig.dual_oriented_edges}).

\begin{figure}[h]
\begin{center}
\includegraphics{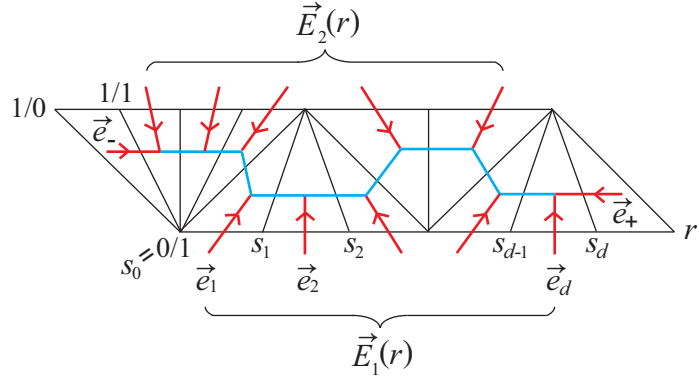}
\end{center}
\caption{\label{fig.dual_oriented_edges}
Dual oriented edges.}
\end{figure}

\begin{proposition}
\label{prop.cusp-shape}
{\rm (1)} The following identity holds:
\[
\sum_{\ee\in \EE_1(r)}\psi_r(\ee)+\sum_{\ee\in \EE_2(r)}\psi_r(\ee)=-1.
\]

{\rm (2)} The cusp shape $\lambda(\OO(r))$ with respect to a suitable choice of
a longitude is given by the following formula:
\[
\frac{1}{2}\lambda(\OO(r))=
\sum_{\ee\in \EE_1(r)}\psi_r(\ee)=
-1-\sum_{\ee\in \EE_2(r)}\psi_r(\ee).
\]
\end{proposition}

\begin{proof}
(1) By Proposition ~\ref{geometric_shape_of_cusp}(1),
$\psi_r(\ee)$ is defined for all $\ee\in \EE(r)$.
Hence, we have the following identity by \cite[Lemma ~1]{Bowditch1}:
\[
\sum_{\ee\in \EE(r)}\psi_r(\ee)=1.
\]
On the other hand, since $\phi_r(\infty)=\phi_r(r)=0$,
we have $\psi_r(\ee_-)=\psi_r(\ee_+)=1$.
Hence, by using the fact that $\EE(r)=\EE_1(r)\sqcup\EE_2(r)\sqcup
\{\ee_-,\ee_+\}$,
we obtain the desired identity.

(2) Let $\ee_1,\ee_2,\dots,\ee_d$ be the members of $\EE_1(r)$
whose heads lie in $\TT_0(r)$ in this order,
and let $s_0,s_1,\dots, s_{d}$ be the vertices of $\Sigma_0(r)$
such that $\ee_i$ is dual to the Farey edge $\langle s_{i-1},s_i\rangle$.
Let $P_0, P_1,\dots, P_d$ be elliptic generators satisfying the following conditions.
\begin{enumerate}[\indent \rm (i)]
\item The slope of $P_i$ is $s_i$ for each $i\in\{0,1,\dots,d\}$.

\item For each $i\in\{1,2,\dots,d\}$,
the two elliptic generators  $P_{i-1}$ and $P_{i}$ appear successively in
the sequence of elliptic generators associated with
the Farey triangle in $\Sigma_0(r)$ which contains $\langle s_{i-1},s_i\rangle$.
\end{enumerate}
Then, by the description of the geometric cusp triangulation in
Proposition ~\ref{geometric_shape_of_cusp},
we see that the zigzag line segment in $\CC$
obtained by joining the points
\[
c(\rho_r(P_0)), c(\rho_r(P_1)), \dots, c(\rho_r(P_d))
\]
projects to a longitude of the orbifold $\OO(r)$
(see Figure ~\ref{fig.actual_cusp_triangulation}).
In particular, $c(\rho_r(P_d))-c(\rho_r(P_0))$
is equal to the complex number $\frac{|K(r)|}{4}\lambda$
introduced at the end of Section ~\ref{Statement}.
Hence the modulus of the cusp of $\OO(r)$
with respect to the longitude is given by
\[
\frac{1}{2}\lambda(\OO(r))
=c(\rho_r(P_d))-c(\rho_r(P_0))
=\sum_{i=1}^d \left(c(\rho_r(P_i)))-c(\rho_r(P_{i-1})\right)
=\sum_{i=1}^d\psi(e_i),
\]
where the last equality follows from formula ~\eqref{a=x/yz}
in Section ~\ref{sec:Markoff}.
Thus we have proved the first identity in (2).
The second identity follows from (1).
\end{proof}

By the above proposition,
the proof of Theorems ~\ref{MainTheorem1} and \ref{MainTheorem2}
is reduced to the following key lemma.

\begin{key-lemma}
\label{key-lemma}
Every member $\ee$ of $\EE_1(r)\cup \EE_2(r)$
satisfies the conditions of Proposition ~\ref{key-proposition},
namely, it satisfies the following conditions.
\begin{enumerate}[\indent \rm (1)]
\item The set $\{s\in\Omega^-(\ee)\,|\,\ |\phi_r(s)|\le 2\}$ is
finite.

\item $\Omega^{0-}(\ee)\cap\phi_r^{-1}(-2,2)=\emptyset$.
\end{enumerate}
\end{key-lemma}

\begin{proof}[Proof of Theorems ~\ref{MainTheorem1} and \ref{MainTheorem2}
assuming Key Lemma ~\ref{key-lemma}]
By Proposition ~\ref{key-proposition} and Lemma ~\ref{key-lemma},
we have the following identity for each $j=1,2$
\begin{align*}
\sum_{\ee\in \EE_j(r)}\psi_r(\ee)
&=\sum_{\ee\in \EE_j(r)} \left\{
\sum_{s\in\Omega^0(e)} h(\phi_r(s))
+2\sum_{s\in\Omega^-(\ee)} h(\phi(s))\right\} \\
&=2\sum_{s\in \mathrm{int}I_j(r)}\frac{1}{1+e^{l_{\rho_r}(\beta_s)}}
+\sum_{s\in\partial I_j(r)}\frac{1}{1+e^{l_{\rho_r}(\beta_s)}}.
\end{align*}
By applying this identity to the identities
in Proposition ~\ref{prop.cusp-shape},
we obtain the desired results.
\end{proof}

Key Lemma ~\ref{key-lemma} is proved by using the results
obtained in the series of papers
\cite{lee_sakuma, lee_sakuma_2, lee_sakuma_3, lee_sakuma_4}
(see also the announcement \cite{lee_sakuma_5}),
which gives a complete answer to the following question
concerning the simple loops in 2-bridge sphere $\PConway$
of a $2$-bridge link $K(r)$.
\begin{enumerate}[\indent \rm (1)]
\item Which simple loop on $\PConway$ is null-homotopic or peripheral
on $S^3-K(r)$?

\item For given two simple loops on $\PConway$,
when are they homotopic?
\end{enumerate}
In particular, we have the following theorem.

\begin{theorem}
\label{prop:conjugacy}
For a hyperbolic $2$-bridge link $K(r)$, the following hold.
\begin{enumerate}[\indent \rm (1)]
\item For any rational number $s$ in $I_1(r)\cup I_2(r)$,
$\alpha_s$ is not null-homotopic in $S^3-K(r)$.

\item There are at most two rational numbers $s$ in $I_1(r)\cup I_2(r)$
such that $\alpha_s$ is peripheral.

\item Except for at most two pairs of rational numbers
in $I_1(r)\cup I_2(r)$,
the simple loops $\{\alpha_s \svert s\in I_1(r)\cup I_2(r)\}$
are not mutually homotopic in $S^3-K(r)$.
\end{enumerate}
\end{theorem}

\begin{proof}[Proof of Key Lemma ~\ref{key-lemma}]
Since $\Omega^{0-}(\ee)\subset I_1(r)\cup I_2(r)$ for any
$\ee\in \EE_1(r)\cup \EE_2(r)$,
the lemma is reduced to the following assertions.
\begin{enumerate}[\indent \rm (i)]
\item
The set $\{s\in I_1(r)\cup I_2(r)\,|\,\ |\phi_r(s)|\le 2\}$ is finite.

\item $(I_1(r)\cup I_2(r))\cap\phi_r^{-1}(-2,2)=\emptyset$.
\end{enumerate}
We first prove (ii).
Let $s$ be a rational number contained in $I_1(r)\cup I_2(r)$.
Then, by Theorem ~\ref{prop:conjugacy}(1),
$\alpha_s$ determines a non-trivial element of $\pi_1(S^3-K(r))$.
Since $\rho_r$ is induced by the holonomy representation
of the complete hyperbolic structure of $S^3-K(r)$,
we see that $\rho_r(\alpha_s)=\rho_r(\beta_s^2)$ is neither trivial nor elliptic.
Thus $\rho_r(\beta_s)$ is not elliptic, and so
$\phi_r(s)=\tr(\tilde\rho_r(\beta_s))$ is not contained in $(-2,2)$.
Hence we obtain (ii).

Next we prove (i).
Suppose on the contrary that
the set $\{s\in I_1(r)\cup I_2(r)\,|\,\ |\phi_r(s)|\le 2\}$ contains
infinitely many elements $\{s_j\}_{j\in\ZZ}$.
By Theorem ~\ref{prop:conjugacy}(1) and (2),
we may assume that $\rho(\alpha_{s_j})$
is neither trivial nor parabolic, and hence,
the simple loop $\alpha_{s_j}$ is homotopic to
a closed geodesic in the hyperbolic manifold $S^3-K(r)$.
By Theorem ~\ref{prop:conjugacy}(3), we may also assume that
$\{\alpha_{s_j}\}$ are not mutually homotopic in $S^3-K(r)$ and so
the corresponding closed geodesics are mutually distinct.
On the other hand, the condition $|\phi(s_j)|\le 2$ implies that
the real length $L(\rho(\alpha_{s_j}))=2L(\rho(\beta_{s_j}))$ is bounded from above.
This contradicts the discreteness of marked length spectrum
of geometrically finite hyperbolic $3$-manifolds
(see Lemma ~\ref{length-specturum} below).
Hence we obtain (i).
This completes the proof of Key Lemma ~\ref{key-lemma}.
\end{proof}

Since we could not find a proof of
Lemma ~\ref{length-specturum} below in a literature,
we give a proof, for completeness, imitating
the argument in \cite[Proof of Theorem ~1 in p.73]{Abikoff},
where we refer to \cite{Matsuzaki-Taniguchi} for terminologies for Kleinian groups.

\begin{lemma}
\label{length-specturum}
Let $M$ be a geometrically finite complete hyperbolic 3-manifold.
Then the marked length spectrum of $M$ is discrete,
namely, for any positive real number $L$,
there are only finitely many closed geodesics in $M$
with length at most $L$.
\end{lemma}

\begin{proof}
Suppose on the contrary that there is an infinite sequence $\{\gamma_j\}$
of mutually distinct closed geodesics in
a geometrically finite complete hyperbolic 3-manifold $M=\HH^3/\pi_1(M)$
such that the lengths $L(\gamma_j)$ are bounded above by a positive constant $L$.
Pick a base point $b$ in the convex core, $C(M)$, of $M$
and pick a lift $\tilde b$ of $b$ in the universal cover $\HH^3$.
Let $b_j$ a point in $\gamma_j$
such that $d(b,b_j)=d(b,\gamma_j)$,
where $d$ denotes the hyperbolic distance.
Let $\tilde b_j\in\HH^3$ be the lift of $b_j$ such that $d(\tilde b,\tilde b_j)=d(b,b_j)$,
and let $\tilde\gamma_j$ be the geodesic in $\HH^3$ passing through $\tilde b_j$
which projects to $\gamma_j$.
We abuse notation to continue to denote by $\gamma_j$ the element of
$\pi_1(M)(\subset \Isom(\HH^3))$
represented by the closed geodesic $\gamma_j$ whose axis is $\tilde\gamma_j$.
Then
\[
d(\tilde b, \gamma_j(\tilde b)) \le 2 d(b,b_j) + L(\gamma_j) \le 2 d(b,b_j)+L.
\]
So we have
\[
d(b,b_j)\ge \frac{1}{2}(d(\tilde b, \gamma_j(\tilde b))-L).
\]
Since $\pi_1(M)$ acts discontinuously on $\HH^3$,
we see $d(\tilde b, \gamma_j(\tilde b))\to \infty$, and hence
the sequence $\{b_j\}$ in $C(M)$ diverges.
Since $M$ is geometrically finite,
the convex core $C(M)$ is a union of a compact submanifold and
a finite union of cusp neighborhoods.
So, we may assume, after taking a subsequence,
that $\{b_j\}$ converges to a cusp of $C$.
Let $U$ be a neighborhood of the cusp in $M$,
obtained as the image of a horoball $H$.
Then the stabilizer of $H$ is a parabolic subgroup, $P_U$, of $\pi_1(M)$
and $H$ is precisely invariant by $(\pi_1(M),P_U)$.
Since $\{b_j\}$ converges to the cusp,
we can find, for a sufficiently large $j$,
a lift $\tilde b_j'\in H$ of $b_j$
such that $d(\tilde b_j',\partial H)>L$.
Let $\gamma_j'$ be the element of $\pi_1(M)$ represented by
the closed geodesic $\gamma_j$
whose axis passes through $\tilde b_j'$.
Then $d(\tilde b_j',\gamma_j'(\tilde b_j'))=l(\gamma_j)\le L$, and therefore
the point $\gamma_j'(\tilde b_j')$ is also contained in $H$.
Since $\gamma_j\in\pi_1(M)$ is not parabolic (and therefore
it does not belong to $P_U$),
this contradicts the assumption that $H$ is
precisely invariant by $(\pi_1(M),P_U)$.
\end{proof}

\section{Homological description of the longitude in Theorem ~\ref{MainTheorem1}}
\label{sec:longitude}

In this section, we give an explicit homological description
of the longitude of $K(r)$
in Theorem ~\ref{MainTheorem1}.
To this end, we fix an arbitral orientation of $K(r)$, and
employ the following notation.
\begin{enumerate}[\indent \rm (1)]
\item If $K(r)$ is a knot, then $\ell_0$ and $m$ denote the
preferred longitude and the meridian of $K(r)$,
namely, $\ell_0$ and $m$ are oriented essential simple loops on
a peripheral torus, satisfying the following conditions:
$\ell_0$ is homologous to $K(r)$ in a regular neighborhood,
$N(K(r))$, of $K(r)$ and $\lk(\ell_0,K(r))=0$;
$m$ bounds a disk in $N(K(r))$ and
$\lk(m, K(r))=1$.
The symbol $\ell$ denotes the longitude of $K(r)$ in Theorem ~\ref{MainTheorem1}, which is oriented so that
it is homologous to $K(r)$ in $N(K(r))$.

\item If $K(r)$ consists of two components, $K_1$ and $K_2$,
then $\ell_{i,0}$ and $m_i$ denotes the preferred longitude and the
meridian of $K_i$ for $i=1,2$.
The symbol $\ell_i$ denotes the longitude of $K_i$ in
Theorem ~\ref{MainTheorem1}, which is oriented so that
it is homologous to $K_i$ in $N(K_i)$  for $i=1,2$.
We denote by $\ell$ the union $\ell_1\cup \ell_2$,
and call it the longitude of $K(r)$ in Theorem ~\ref{MainTheorem1}.
\end{enumerate}

We express $\ell$ in terms of $\ell_0$ and $m$ when $K(r)$
is a knot, and express $\ell_i$ in terms of
$\ell_{i,0}$ and $m_i$ when $K(r)$ is a two-component link.
To this end, we use the alternating diagram of $K(r)$
associated with the continued fraction expansion
described at the beginning of Section ~\ref{sec:canonical},
and consider the decomposition of $(S^3,K(r))$ into $n+2$ blocks,
$\{\block_k\}_{k=0}^{n+1}$,
as illustrated in Figure ~\ref{fig.decomposition_link_diagram}.
The {\it first block} $\block_0$ and the {\it last block} $\block_{n+1}$
are upper and lower bridges, respectively, and
a {\it middle block} $\block_k$ ($1\le k \le n$) is a $4$-braid,
where the second and the third strings
form $a_k$ right hand half-twists when $k$ is odd;
the first and the second strings
form $a_k$ left hand half-twists when $k$ is even.
For each middle block $\block_k$,
we define $\delta_k$ to be $0$ or $1$ according to whether
the orientations of the strings forming the half-twists in $\block_k$
are parallel or not (see Figure ~\ref{fig.linking_number}).
Then the longitude $\ell$ in Theorem ~\ref{MainTheorem1}
is given by the following proposition.

\begin{figure}[h]
\begin{center}
\includegraphics{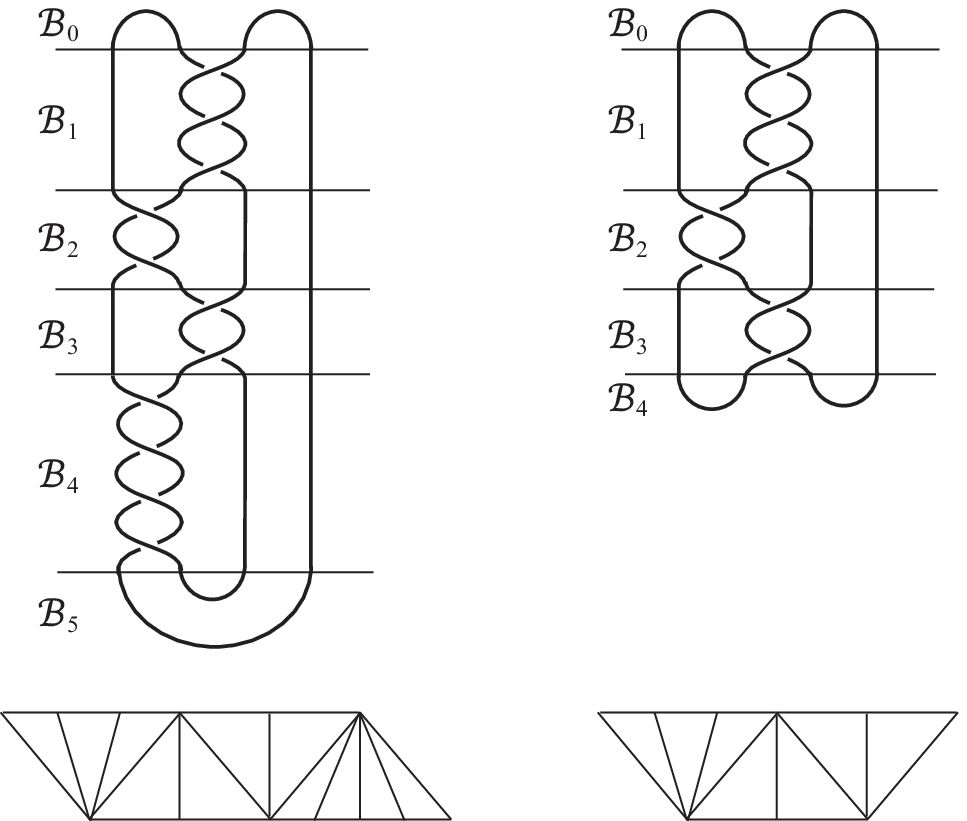}
\end{center}
\caption{\label{fig.decomposition_link_diagram}
The decomposition $\{\block_k\}_{k=0}^{n+1}$ of $(S^3,K(r))$.}
\end{figure}

\begin{proposition}
\label{prop:longitude}
Under the above notation,
the linking number $\lk(\ell,K(r))$,
for the longitude $\ell$ in Theorem ~\ref{MainTheorem1},
is given by the following formula:
\[
\lk(\ell, K(r)) =
\begin{cases}
2\left(\sum_{k=1}^n \delta_k (-1)^{k-1}a_k\right) & \mbox{if $n$ is odd,}\\
2\left(1+\sum_{k=1}^{n} \delta_k (-1)^{k-1} a_k\right) & \mbox{if $n$ is even.}
\end{cases}
\]
From this linking number, $\ell$ is determined by the following formula.
\begin{enumerate}[\indent \rm (1)]
\item If $K(r)$ is a knot, then
\[
[\ell]=[\ell_0]+\lk(\ell,K(r))[m] \in H_1(\partial N(K(r))).
\]

\item If $K(r)$ is a link, then for $i=1,2$,
\[
[\ell_i]=[\ell_{i,0}]+\frac{1}{2}(\lk(\ell,K(r))-2\lk(K_1,K_2))[m_i]
\in H_1(\partial N(K_i)).
\]
\end{enumerate}
\end{proposition}

\begin{remark}
\label{remark:prop:longitude}
{\rm
If every $a_i$ is even and $n$ is odd, then $K(r)$ is a knot
and the longitude $\ell$ is given by the following simple formula:
\[
[\ell]=[\ell_0]+2[m].
\]
}
\end{remark}

Our task is to draw the longitude $\ell$ on the boundary of
the link exterior $E(K(r)):=S^3-\interior N(K(r))$.
To this end,
let $\{\eblock_k\}_{k=0}^{n+1}$ be the decomposition of
the knot exterior $E(K(r))$
induced by the decomposition $\{\block_k\}_{k=0}^{n+1}$ of $(S^3,K(r))$.
We denote by $\partial_p \eblock_k$ the intersection of $\partial E(K(r))$
with $\partial \eblock_k$.
Then $\partial_p \eblock_k$ is a disjoint union of four or two annuli
according to whether $k\in \{1,2,\dots, n\}$ or $k\in \{0,n+1\}$.

Set $c_0=0$, $c_k=\sum_{i=1}^k a_i$ and
$\sigma_i^{(k)}=\sigma_{c_{k-1}+i}$ ($1\le i\le a_k)$.
Then the $k$-th middle block $\block_k$  ($1\le k\le n$) corresponds to the subchain
\[
\Sigma^{(k)}(r):=(\sigma_1^{(k)}, \sigma_2^{(k)}, \dots,\sigma_{a_k}^{(k)})
\]
of $\Sigma(r)$.
The subchain $\Sigma^{(k)}(r)$ contains a unique {\it pivot}, $s^{(k)}_*$,
i.e., a vertex of $\Sigma(r)$
which is shared by at least three triangles in $\Sigma(r)$.
We denote the remaining vertices of $\Sigma^{(k)}(r)$
by $s^{(k)}_0, s^{(k)}_1, \dots, s^{(k)}_{a_k}$,
where they sit in $\Sigma^{(k)}(r)$ in this order.
The Farey triangle $\sigma_i^{(k)}$ is spanned by the vertices
$\{s^{(k)}_*,s^{(k)}_{i-1},s^{(k)}_i\}$.
By the description of $\DD(r)$ as a quotient of $\hat\DD(r)$ in Section ~\ref{sec:canonical},
we may regard $\hat\DD(r)$
as a decomposition of the middle part
$\cup_{i=1}^k\eblock_k$ into truncated ideal tertrahedra:
the collapsing of the edges of slopes $\infty$ and $r$
has the effect of adding the first and the last blocks
$\eblock_0$ and $\eblock_{n+1}$.
To be precise, the following hold.
\begin{enumerate}[\indent \rm (1)]
\item For each $k\in\{1,2,\dots, n\}$,
the pair of edges of the pivot slope $s^{(k)}_*$ is represented by
the pair of straight horizontal arcs
in the upper boundary of the $k$-th middle block
$\eblock_k$ as illustrated in Figure ~\ref{fig.middle_block}.
It is isotopic to the
straight horizontal arcs
in the lower boundary of $\eblock_k$
as illustrated in Figure ~\ref{fig.middle_block}.
An isotopy between these two representatives determines
four mutually disjoint arcs, $\xi_k$,
properly embedded in $\partial_p \eblock_k$.
(To be precise, $\xi_k$ is the intersection with $\partial_p \eblock_k$ of
two properly embedded disks in $\eblock_k$ realizing the isotopy.)

\item For each $k\in\{1,2,\dots, n\}$ and $i\in \{0,1,\dots, a_k\}$,
the pair of edges of slope $s^{(k)}_i$ are represented by the pair of
straight horizontal arcs
which sit in the level sphere between
the $i$-th and $i+1$-th
crossings in $\block_k$,
as illustrated in Figure ~\ref{fig.middle_block}.
Each of the four corners of the ideal triangles in $\trg(\sigma_i^{(k)})$
bounded by an edge of slope $s^{(k)}_{i-1}$ and an edge of slope $s^{(k)}_{i}$
determines an arc in $\partial_p \eblock_k$
joining an end point of an edge of slope $s^{(k)}_{i-1}$ and
an end point of an edge of slope $s^{(k)}_{i}$.
The union of these arcs, where
$\sigma_i^{(k)}$ runs over the Farey triangles of $\Sigma^{(k)}(r)$,
forms four properly embedded arcs in $\partial_p\eblock_k$,
which is disjoint from $\xi_k$.
We denote by $\eta_k$ the union of these four arcs in $\partial_p\eblock_k$.

\item
In the top block $\eblock_0$,
each of the two edges of slope $\infty$ in $\eblock_0\cap\eblock_1$
shrinks to a point,
and the two edges of slope $0$ in $\eblock_0\cap\eblock_1$
are isotopic to the upper tunnel.
Let $\xi_0$ be the pair of arcs properly embedded in $\partial_p \eblock_0$
determined by the isotopy (see Figure ~\ref{fig.top_block}).

\item
In the bottom block
$\eblock_{n+1}$,
each of the two edges of slope $r$ in $\eblock_n\cap\eblock_{n+1}$
shrinks to a point,
and the two edges of slope $[a_1,a_2,\dots,a_{n-1}]$
in $\eblock_n\cap\eblock_{n+1}$
are isotopic to the lower tunnel.
Let $\xi_{n+1}$ be the pair of arcs properly embedded in $\partial_p \eblock_k$
determined by the isotopy.
\end{enumerate}
To be more precise,
the (actual) cusp torus is identified with a quotient of the
``model'' torus $\cup_{k=0}^{n+1}\partial_p \eblock_k$,
where each arc component of
$\xi_k$ ($k=0,\dots, n+1$) shrinks to a point.

\begin{figure}[h]
\begin{center}
\includegraphics{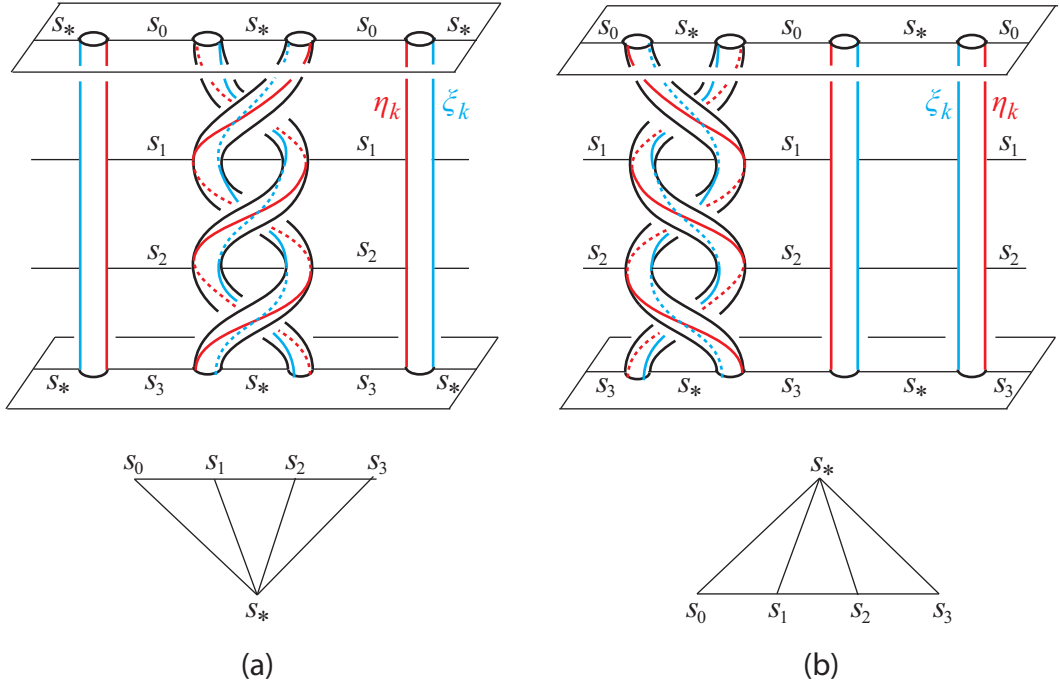}
\end{center}
\caption{\label{fig.middle_block}
The middle block $\eblock_k$ (a) when $k$ is odd, and (b) when $k$ is even.
The labels $s_*$ and $s_j$ for the straight horizontal arcs
are the slopes of the edges,
where $s_*$ and $s_j$ are abbreviations for $s_*^{(k)}$ and $s_j^{(k)}$,
respectively.
The set $\xi_k$ consists of four vertical arcs joining the endpoints of the edges
of slope $s_*^{(k)}$ in the upper boundary of $\eblock_k$
and the endpoints of the edges
of slope $s_*^{(k)}$ in the lower boundary of $\eblock_k$.
The set $\eta_k$ consists of four vertical arcs joining the endpoints of
the edges
of slope $s_0^{(k)}$ in the upper boundary of $\eblock_k$
and the endpoints of the edges
of slope $s_{a_k}^{(k)}$ in the lower boundary of $\eblock_k$,
passing through the endpoints of the edges
of slope $s_i^{(k)}$ ($0<i<a_k$).
}
\end{figure}

\begin{figure}[h]
\begin{center}
\includegraphics{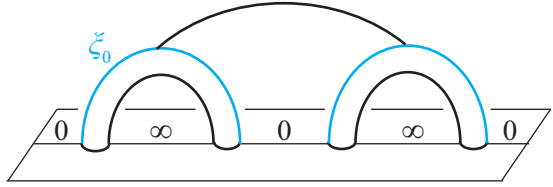}
\end{center}
\caption{\label{fig.top_block}
The top block $\eblock_0$.
The set $\xi_0$ is the union of the two arcs on $\partial_p \eblock_0$
joining the endpoints of the edges of slope $0$ in the lower boundary
$\eblock_0\cap\eblock_1$ of $\eblock_0$.
}
\end{figure}

\begin{lemma}
\label{lem:logitude-picture1}
Suppose that $n$ is an odd integer $2n_0+1$.
Then under the above identification of the cusp torus
with a quotient of the model torus $\cup_{k=0}^{n+1}\partial_p \eblock_k$,
the longitude $\ell$ is equal to {\rm (}the image of{\rm )} the union of
the following arcs (see Figure ~\ref{fig.longitude}):
\[
\xi_0\cup \xi_1,\ \eta_2,\ \xi_3,\ \eta_4,\ \cdots,\
\xi_{2n_0-1},\ \eta_{2n_0},\ \xi_{2n_0+1}\cup \xi_{2n_0+2}.
\]
\end{lemma}

\begin{proof}
Let $s_0,s_1,\dots, s_d$ be the slopes in
the proof of Proposition ~\ref{prop.cusp-shape}.
Recall that the longitude of $\partial \OO(r)$ is obtained as the image
of the zigzag line in $\CC$ spanned by
$c(\rho_r(P_0)), c(\rho_r(P_1)), \dots, c(\rho_r(P_{d}))$,
where $P_i$ is an elliptic generator with $s(P_i)=s_i$
as in the proof of Proposition ~\ref{prop.cusp-shape}.
We note that the following hold.
\begin{enumerate}[\indent \rm (1)]
\item The set $\{s_0,s_1,\dots, s_d\}$ consists of
the vertices of a subchain $\Sigma^{(2k)}(r)$
different from the pivot $s^{(2k)}_*$,
where $2k$ runs over the set $\{2,4,\dots,2n_0\}$.

\item The first slope $s_0$ is equal the pivot $s^{(1)}_*$ of the subchain $\Sigma^{(1)}(r)$.

\item The last vertex $s_d$ is equal the pivot $s^{(n)}_*$ of the subchain $\Sigma^{(n)}(r)$.
\end{enumerate}
In particular, the first $(a_2+1)$ vertices $s_0, s_1, \dots, s_{a_2}$ are equal to the
vertices $s^{(2)}_0, s^{(2)}_1, \dots, s^{(2)}_{a_2}$ of the subchain $\Sigma^{(2)}(r)$
different from the pivot $s^{(2)}_*$.
Thus the sub zigzag line in $\CC$
spanned by $c(\rho_r(P_0)), c(\rho_r(P_1)), \dots, c(\rho_r(P_{a_1}))$
corresponds to (a component of) $\eta_2\subset\partial_p \eblock_2$.
Similarly, for each even integer $2k\in \{2,4,\dots,2n_0\}$,
the sub zigzag line in $\CC$
spanned by $c(\rho_r(P_i))$'s, where $s(P_i)$ runs over the vertices of
the subchain $\Sigma^{(2k)}(r)$ different from the pivot $s^{(2k)}_*$,
corresponds to
(a component of) $\eta_{2k}\subset\partial_p \eblock_{2k}$.

For each $2k\in \{2,4,\dots,2n_0-2\}$, the final slope $s^{(2k)}_{a_{2k}}$
of the subchain $\Sigma^{(2k)}(r)$  is equal to the first slope
$s^{(2k+2)}_{0}$ of the subchain $\Sigma^{(2k+2)}(r)$,
and they are equal to the pivot
$s^{(2k+1)}_*$ of the subchain $\Sigma^{(2k+1)}(r)$.
Thus  each edge of slope $s^{(2k)}_{a_{2k}}$,
which lies in the lower boundary of $\partial_p\eblock_{2k}$,
is isotopic in $\eblock_{2k+1}$ to an
edge of slope $s^{(2k+2)}_{0}$,
which lies in the upper boundary of the block $\eblock_{2k+2}$.
So, in order to describe the longitude
in the model torus $\cup_{k=0}^{n+1}\partial_p \eblock_k$,
the end points of $\eta_{2k}$ in
the lower boundary of $\partial_p \eblock_{2k}$
should be joined with
the end points of $\eta_{2k+2}$ in
the upper boundary of $\partial_p \eblock_{2k+2}$
by the trace, $\xi_{2k+1}\subset \partial_p \eblock_{2k+1}$,
of the isotopy (see Figure ~\ref{fig.longitude}).

The first slope $s^{(2)}_{0}$ of the subchain
$\Sigma^{(2)}(r)$ is equal to the
the pivot $s^{(1)}_*$ of the subchain $\Sigma^{(1)}(r)$.
Thus the edges of slope $s^{(2)}_{0}$,
which lie in the upper boundary of the block $\eblock_{2}$,
are equal to the edges of slope $s^{(1)}_*$
which lie in the lower boundary of the block $\eblock_{1}$.
The latter edges are isotopic in $\eblock_{1}$ to the edges of slope $s^{(1)}_*$
which lie in the upper boundary of $\eblock_{1}$.
These in turn are isotopic in $\eblock_{0}$ to the upper bridges.
So, the endpoints of  $\eta_2$ in the upper boundary
are joined each other
by the trace, $\xi_0\cup\xi_1 \subset \partial_p\eblock_{0}\cup \partial_p\eblock_{1}$,
of the isotopy (see Figure ~\ref{fig.longitude}).
Similarly,
the end points of $\eta_{2n_0}$ in
the lower boundary of $\partial_p \eblock_{2n_0}$
should be joined by the trace, $\xi_{2n_0+1}\cup\xi_{2n_0+2}
\subset \partial_p\eblock_{2n_0+1}\cup \partial_p\eblock_{2n_0+2}$,
of the isotopy.
Hence, the longitude $\ell$ is as described in the lemma.
\end{proof}

\begin{figure}[h]
\begin{center}
\includegraphics{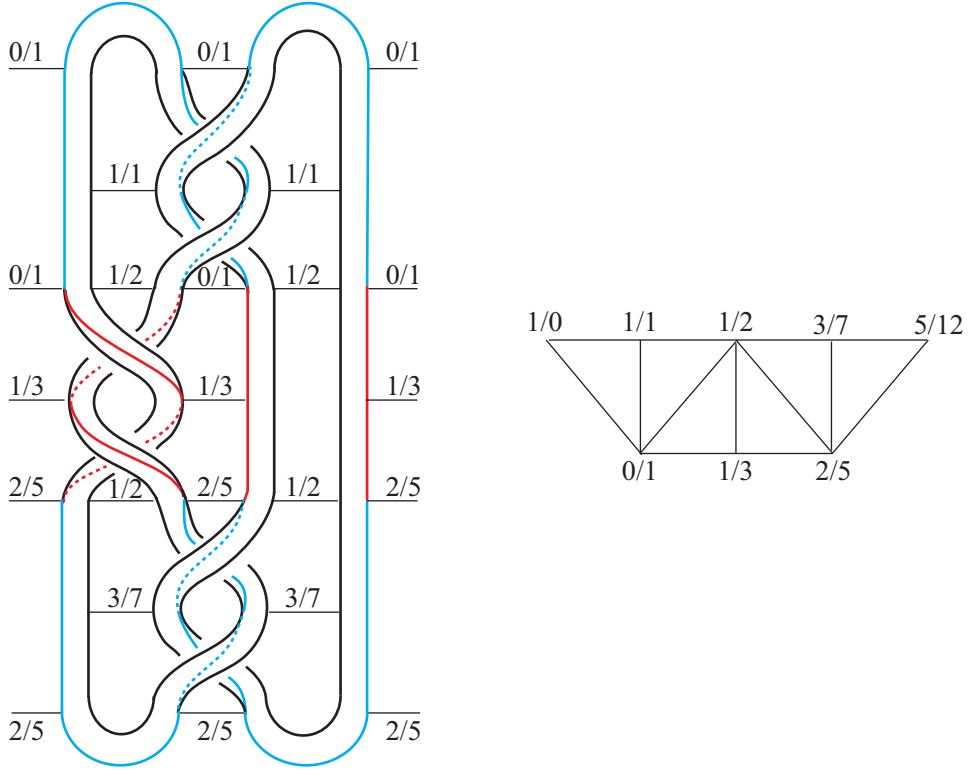}
\end{center}
\caption{\label{fig.longitude}
The whole picture of the longitude $\ell$ when $n=3$.
In this case, $\ell$ is the union of
$\xi_0\cup \xi_1$, $\eta_2$ and $\xi_{4}\cup \xi_{4}$.
}
\end{figure}

In order to treat the case when
$n$ is an even integer $2n_0$,
we need to introduce the following notation (see Figure ~\ref{fig.bottom_block}).
\begin{enumerate}[\indent \rm (1)]
\item By the symbol $\eta_n'$,
we denote the sub arcs of $\eta_n$
bounded by the end points of $\eta_n$ in the upper boundary of
$\partial_p\eblock_{n}$
and the endpoints of the edges of slope
$r':=[a_1,a_2,\dots, a_{n-1}, a_n-1]$.

\item The edges of slope $r'$ in $\eblock_{n}$ is isotopic in
$\eblock_{n}\cup \eblock_{n+1}$ to the lower tunnel.
The symbol $\xi_{n+1}'$ denotes the arcs in
$\partial_p\eblock_{n}\cup \partial_p\eblock_{n+1}$
obtained from the isotopy.
\end{enumerate}

\begin{figure}[h]
\begin{center}
\includegraphics{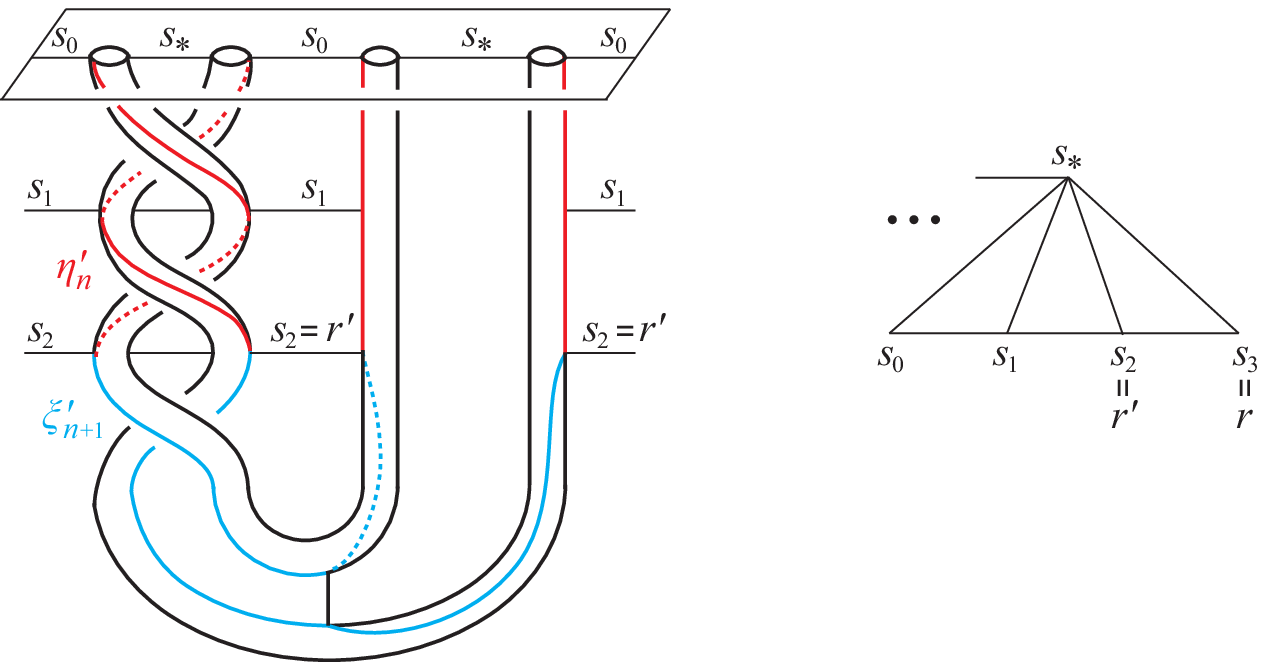}
\end{center}
\caption{\label{fig.bottom_block}
$\eblock_{n} \cup \eblock_{n+1}$
when $n$ is even.}
\end{figure}

By the proof of Lemma ~\ref{lem:logitude-picture1},
we obtain the following lemma.

\begin{lemma}
\label{lem:logitude-picture2}
Suppose that $n$ is an even integer $2n_0$.
Then under the identification of the cusp torus
with a quotient of the model torus $\cup_{k=0}^{n+1}\partial_p \eblock_k$,
the longitude $\ell$ is equal to {\rm (}the image of{\rm )} the union
of the arcs
\[
\xi_0\cup \xi_1,\ \eta_2,\ \xi_3,\ \eta_4,\ \cdots,\
\xi_{2n_0-1},\ \eta_{2n_0}',\  \xi_{2n_0+1}'.
\]
\end{lemma}

\begin{proof}[Proof of Proposition ~\ref{prop:longitude}]
We evaluate the linking number $\lk(\ell, K(r))$
by counting the number of crossings with sign
where $\ell$ runs below $K(r)$.
Then the formula for the linking number follows from the following observations.
\begin{enumerate}[\indent \rm (i)]
\item If $2\le k\le n-1$, then
we can easily check that the contribution of the crossings in
$\block_k$ is equal to $2\delta_k (-1)^{k-1}a_k$
(see Figure ~\ref{fig.linking_number}).

\item The contribution of the crossings in
$\block_0\cup \block_1$ is equal to $2\delta_k (-1)^{k-1}a_k$
with $k=1$.

\item The contribution of the crossings in
$\block_n\cup \block_{n+1}$ is equal to
$2\delta_n (-1)^{n-1}a_n$ or
$2(\delta_n (-1)^{n-1} a_n +  1)$
according to whether $n$ is odd or even.
\end{enumerate}

The formula for the longitude obvious if $K(r)$ is a knot.
Suppose $K(r)=K_1\cup K_2$.
Then by using the symmetry, we have the following identity
for $\{i,j\}=\{1,2\}$.
\begin{align*}
\lk(\ell,K(r))
&=\lk(\ell_1\cup\ell_2,K_1\cup K_2)\\
&=2\lk(\ell_i,K_i)+2\lk(\ell_i,K_j)\\
&=2\lk(\ell_i,K_i)+2\lk(K_1,K_2)
\end{align*}
Hence we obtain the desired formula for $\ell_i$.
\end{proof}

\begin{figure}[h]
\begin{center}
\includegraphics{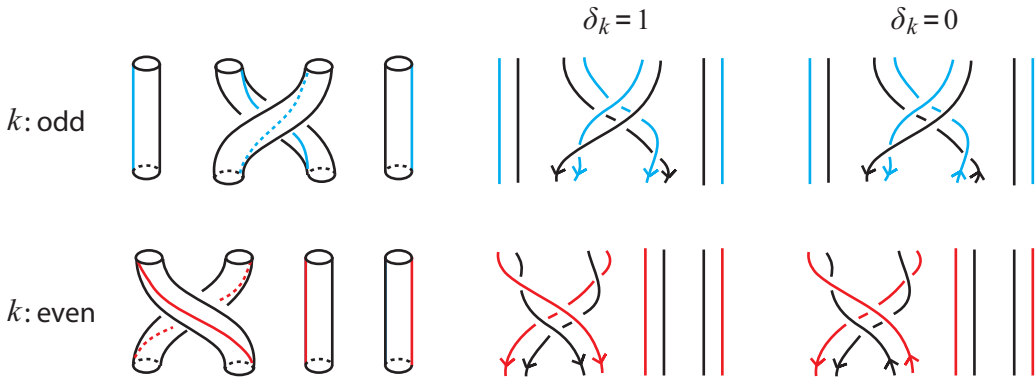}
\end{center}
\caption{\label{fig.linking_number}
If $\delta_k=1$, then the contribution of each of the crossings in
$\block_k$ to the linking number is equal to
$2$ or $-2$ according as $k$ is odd or even.
If $\delta_k=0$, then there is no contribution.}
\end{figure}

\section{Application to end invariants of $SL(2,\CC)$-characters}
\label{sec:end_invarinat}

By extending the concept of a geometrically infinite end
of a Kleinian group,
Bowditch ~\cite{Bowditch2}
introduced the notion of the end invariants of a type-preserving
$SL(2,\CC)$-representation of $\pi_1(\ptorus)$.
Tan, Wong and Zhang ~\cite{Tan_Wong_Zhang_1, Tan_Wong_Zhang_6} extended this notion (with slight modification)
to an arbitrary $SL(2,\CC)$-representation of $\pi_1(\ptorus)$.

To recall the definition of end invariants,
let $\curve$ be the set of free homotopy classes
of essential simple loops on $\ptorus$.
Then as described in Section ~\ref{Statement},
$\curve$ is identified with $\QQQ$, the vertex set of
the Farey tessellation $\DD$ by the rule
$s\mapsto \beta_s$.
The projective lamination space $\plamination$ of $\ptorus$
is then identified with $\RRR:=\RR\cup\{\infty\}$
and contains $\curve$ as the dense subset
of rational points.

\begin{definition}
\label{def_end_invariant}
{\rm
Let $\rho$ be an $SL(2,\CC)$-representation of $\pi_1(\ptorus)$.

(1) An element $X\in \plamination$ is an {\it end invariant}
of $\rho$ if there exists a sequence of distinct elements
$X_n\in\curve$ such that
\begin{enumerate}[\indent \rm (i)]
\item $X_n\to X$, and

\item $\{|\tr \rho(X_n)|\}_n$ is bounded from above.
\end{enumerate}

(2) $\einv(\rho)$ denotes the set of end invariants of $\rho$.
}
\end{definition}

Note that $\einv(\rho)$ is actually determined by the $PSL(2,\CC)$-representation
induced by $\rho$,
because $|\tr \rho(X_n)|$ is determined by the $PSL(2,\CC)$-representation.
Note also that the condition that
$\{|\tr \rho(X_n)|\}_n$ is bounded from above
is equivalent to the condition
that the (real) hyperbolic translation lengths of
the isometries $\rho(X_n)$ of $\HH^3$
are bounded from above.

Tan, Wong and Zhang ~\cite{Tan_Wong_Zhang_1, Tan_Wong_Zhang_6}
showed that $\einv(\rho)$
is a closed subset of $\plamination$ and
proved various interesting properties of $\einv(\rho)$,
including a characterization of
those representations $\rho$
with $\einv(\rho)=\emptyset$ or $\plamination$,
generalizing corresponding results of Bowditch ~\cite{Bowditch2}
for type-preserving representations.
They also proposed an interesting conjecture
\cite[Conjecture ~1.8]{Tan_Wong_Zhang_6}
concerning possible homeomorphism types of $\einv(\rho)$.
The following is a modified version of the conjecture
which Tan ~\cite{Tan} informed to the authors.

\begin{conjecture}
\label{conj_TWZ}
{\rm
Suppose $\einv(\rho)$ has at least two accumulation points.
Then $\einv(\rho)$ is either a Cantor set of $\plamination$
or all of $\plamination$.
}
\end{conjecture}

They constructed a family of representations $\rho$
which have Cantor sets as $\einv(\rho)$,
and proved the following supporting evidence of the conjecture
(see \cite[Theorem ~1.7]{Tan_Wong_Zhang_6}).

\begin{theorem}
\label{discrete_TWZ}
Let $\rho:\pi_1(\ptorus)\to SL(2,\CC)$ be {\rm discrete}
in the sense that the set
$\{\tr(\rho(X))\svert X\in\curve\}$
is discrete in $\CC$.
Then if $\einv(\rho)$ has at least three elements,
then $\einv(\rho)$ is either a Cantor set of $\plamination$
or all of $\plamination$.
\end{theorem}

The above theorem together with Lemma ~\ref{length-specturum}
implies that $\einv(\rho_r)$
of the representation $\rho_r$ induced by
the holonomy representation of
a hyperbolic $2$-bridge link $K(r)$
is a Cantor set.
But it does not give us the exact description of $\einv(\rho_r)$.
By using the proof of the main results in Section ~\ref{sec:proof},
we can explicitly determine the end invariants $\einv(\rho_r)$.
To state the theorem, recall that
the {\it limit set} $\Lambda(\RGPP{r})$
of the group $\RGPP{r}$
is the set of accumulation points in the closure of $\HH^2$
of the $\RGPP{r}$-orbit of a point in $\HH^2$.

\begin{theorem}
\label{MainTheorem3}
For a hyperbolic $2$-bridge link $K(r)$,
the set of end invariants $\einv(\rho_r)$
of the holonomy representation $\rho_r$
is equal to the limit set $\Lambda(\RGPP{r})$
of the group $\RGPP{r}$.
\end{theorem}

\begin{proof}
Since $\phi_r(\infty)=\phi_r(r)=0$,
we see that
both $\infty$ and $r$ belongs to $\einv(\rho_r)$
(cf. \cite[Lemma ~3.5(b)]{Tan_Wong_Zhang_6}).
By Therorem \ref{previous_results}(1),
$\einv(\rho_r)$ is invariant by $\RGPP{r}$,
and so it is an
$\RGPP{r}$-invariant closed set.
Thus $\einv(\rho_r)$ contains the closure
of the $\einv(\rho_r)$-orbit of $\infty$ and $r$.
Since $\Lambda(\RGPP{r})$ is the smallest non-empty
$\RGPP{r}$-invariant closed set,
$\einv(\rho_r)$ must contain $\Lambda(\RGPP{r})$.
On the other hand, by the proof of Key Lemma ~\ref{key-lemma},
we see that $\einv(\rho_r)$ is disjoint from $I_1(r)\cup I_2(r)$.
Since $I_1(r)\cup I_2(r)$ is a fundamental domain of the action
of $\Lambda(\RGPP{r})$ on the domain of discontinuity,
$\RRR-\Lambda(\RGPP{r})$,
and since $\einv(\rho_r)$ is a $\RGPP{r}$-invariant,
we see that $\einv(\rho_r)$ is disjoint from the
$\RRR-\Lambda(\RGPP{r})$,
i.e., $\einv(\rho_r)\subset \Lambda(\RGPP{r})$.
Hence we have $\einv(\rho_r)=\Lambda(\RGPP{r})$.
\end{proof}

In Bowditch's original definition of the ``set of end invariants'' ~\cite[p.729]{Bowditch2},
the accidental parabolics are also regarded as an end invariant.
(He denotes the set by the symbol $L(\phi)$,
where $\phi$ is a Markoff map.)
By using the classification of
the essential simple loops in the $2$-bridge sphere
which are peripheral in hyperbolic $2$-bridge links complements
(see \cite{lee_sakuma_3, lee_sakuma_4} and
\cite[Theorem ~2.6(1)]{lee_sakuma_5}),
we have the following theorem
for Bowdich's set of end invariants $L(\rho_r):=L(\phi_r)$.

\begin{theorem}
\label{MainTheorem4}
For a hyperbolic $2$-bridge link $K(r)$ with $0<r<1/2$,
Bowditch's set of end invariants $L(\rho_r)$
of the holonomy representation $\rho_r$
is equal to the limit set $\Lambda(\RGPP{r})$
of the group $\RGPP{r}$, except for the following cases.
\begin{enumerate}[\indent \rm (1)]
\item
If $r=2/5$, then
$L(\rho_r)=\Lambda(\RGPP{r})\cup \RGPP{r}\{1/5,3/5\}$.
\item
If $r=n/(2n+1)$ for some integer $n\ge 3$,
then
$L(\rho_r)=\Lambda(\RGPP{r})\cup \RGPP{r}\{(n+1)/(2n+1)\}$.
\item
If $r=2/(2n+1)$ for some integer $n\ge 3$,
then
$L(\rho_r)=\Lambda(\RGPP{r})\cup \RGPP{r}\{1/(2n+1)\}$.
\end{enumerate}
In the exceptional cases,
$L(\rho_r)$ is the union of the Cantor set $\Lambda(\RGPP{r})$ and
infinitely many isolated points.
\end{theorem}

At the end of this section, we would like to propose the following conjecture,
which is a variation of a special case of \cite[Question D]{Bowditch2}
and \cite[Conjecture 1.9]{Tan_Wong_Zhang_6}.

\begin{conjecture}
\label{conjecture1}
{\rm
Let $\rho:\pi_1(\ptorus)\to PSL(2,\CC)$ be
a type-preserving representation
such that $\einv(\rho)=\Lambda(\RGPP{r})$.
Then $\rho$ is conjugate to the representation $\rho_r$
induced by the holonomy representation of
a hyperbolic $2$-bridge link $K(r)$.
}
\end{conjecture}

\section{Further discussion}
\label{sec:discussion}

As noted in Section ~\ref{sec:canonical},
in the second author's joint work
with Akiyoshi, Wada and Yamashita ~\cite{ASWY},
it is announced that
that there is a continuous family of hyperbolic cone manifolds
$\{M(r;\theta^-,\theta^+)\}_{0\le \theta^{\pm}\le \pi}$
satisfying the following conditions
(see Preface, in particular Figures ~0.22--0.26, of \cite{ASWY}
and the demonstration in  Wada's software OPTi ~\cite{Wada}).
\begin{enumerate}[\indent \rm (1)]
\item The underlying space of  $M(r;\theta^-,\theta^+)$ is $S^3-K(r)$.

\item The cone axis of $M(r;\theta^-,\theta^+)$ consists of
the core tunnel of $(B^3,t(\infty))$ and that of $(B^3,t(r))$,
where the cone angles are $2\theta^-$ and $2\theta^+$, respectively.
In particular, $M(r;\pi,\pi)$ is the complete hyperbolic manifold $S^3-K(r)$.

\item If $0<\theta^{\pm}<\pi$,
then the combinatorial dual of the ``Ford domain'' of $M(r;\theta^-,\theta^+)$
is homeomorphic to $\hat\DD(r)$.
If $\theta^-=\theta^+=\pi$,
the combinatorial dual of the Ford domain of
$M(r;\pi,\pi)=S^3-K(r)$
is homeomorphic to $\DD(r)$,
i.e., the canonical decomposition of $S^3-K(r)$ is homeomorphic to $\DD(r)$.
\end{enumerate}
Thus the announcement in \cite{ASWY} says that
the collapsing of the edges of slopes $\infty$ and $r$ in $\hat\DD(r)$
is realized geometrically by a continuous family of hyperbolic cone manifolds.

Akiyoshi and the second author tried to prove
the main results in this paper,
by establishing the following natural generalization of
Theorem ~\ref{MainTheorem3}.

\begin{conjecture}
\label{conjecture1}
{\rm
Let $\rho=\rho_{(r;\theta^-,\theta^+)}:\pi_1(\OO)\to PSL(2,\CC)$
be the type-preserving
$PSL(2,\CC)$-representation induced by the holonomy representation
of the hyperbolic cone manifold $M(r;\theta^-,\theta^+)$ with
$0\le \theta_{\pm} \le \pi$.
Then the set $\einv(\rho)$ is disjoint from the
fundamental intervals $I_1(r)\cup I_2(r)$.
}
\end{conjecture}

Consider the subset, $J$, of $[0,\pi]\times [0,\pi]$
consisting of those points $(\theta^-,\theta^+)$ for which
the conjecture is valid.
It is obvious that $(0,0)$ belongs to $J$ and so
$J$ is non-empty.
Tan pointed out that \cite{Tan_Wong_Zhang_1}
implies that the set $J$ is open.
So what we need to show is that $J$ is closed.
Though computer experiments seem to support the conjecture,
the conjecture is still open.

\bibstyle{plain}

\end{document}